\documentclass[12pt]{amsart}
\title[Tower of fully commutative elements of type $\tilde C$]{On the fully commutative 
elements of type $\tilde C$  \\ 
 and  faithfulness of related towers}
\author{Sadek AL HARBAT}
\address{} 
\email{sadikharbat@math.univ-paris-diderot.fr}

\usepackage{etex} 
\usepackage[latin1]{inputenc}	
\usepackage[T1]{fontenc}
\usepackage[english]{babel}
\usepackage{amsmath,amssymb}
\usepackage{wasysym}
\usepackage{stmaryrd}
\usepackage[table]{xcolor}
\usepackage{color}
\usepackage{dsfont}\let\mathbb\mathds
\usepackage[all]{xy}
\usepackage{graphicx}
\usepackage{mathenv}
\usepackage{lastpage}
\usepackage{fancyhdr}	
\usepackage{subfigure}
\usepackage{geometry}
\usepackage[version=3]{mhchem}
\usepackage[force,almostfull]{textcomp}
\usepackage{array}
\usepackage{lipsum}		
\usepackage{eso-pic}
\usepackage{multirow}
\usepackage{fancybox}
\usepackage{cite}
\usepackage{amsthm}
\usepackage{listings}
\usepackage{lscape}
\usepackage{multicol}		
\usepackage{setspace}
\usepackage{times}
\usepackage{eurosym}
\usepackage{latexsym}
\usepackage{ulem}
\usepackage{lmodern}
\usepackage{colortbl}
\usepackage{rotating}
\usepackage{type1cm}
\usepackage{lettrine}
\usepackage{ragged2e}
\usepackage{mathrsfs}
\usepackage{scrtime}
\usepackage{enumitem}

\usepackage{longtable}
\usepackage{url}
\usepackage{nomentbl}
\usepackage{nomencl}
\usepackage{hyperref}
\hypersetup{
pdfpagemode=UseOutlines,      
pdfstartview=Fit,             
pdffitwindow=true,            
pdfpagelayout=TwoColumnsRight,
pdftoolbar=true,              
pdfmenubar=true,              
bookmarksopen=false,          
bookmarksnumbered=true,       
colorlinks=true,              
pdfauthor={Ton nom},          
pdftitle={Titre PDF},         
pdfcreator=PDFLaTeX,          %
pdfproducer=PDFLaTeX,         %
linkcolor=blue,               
urlcolor=blue,                
anchorcolor=black,            
citecolor=blue,               
frenchlinks=true,             
pdfborder={0 0 0}             
}
\usepackage{pgf, tikz}
\usetikzlibrary{matrix,positioning}
\usetikzlibrary{arrows,decorations.markings}

\newtheorem{theorem}{Theorem}[section]

\newtheorem{definition}[theorem]{Definition}
\newtheorem{proposition}[theorem]{Proposition}
\newtheorem{lemma}[theorem]{Lemma}
\newtheorem{corollary}[theorem]{Corollary}
\newtheorem{example}[theorem]{Example}
\newtheorem{remark}[theorem]{Remark}

\addto\captionsfrenchb{}
\addto\captionsfrenchb{}

    \newlength{\myarrowsize} 
    \newlength{\myoldlinewidth}

    \pgfarrowsdeclare{myto}{myto}{
        \pgfsetdash{}{0pt} 
        \pgfsetbeveljoin 
        \pgfsetroundcap 
        \setlength{\myarrowsize}{0.5pt}
        \addtolength{\myarrowsize}{.5\pgflinewidth}
        \pgfarrowsleftextend{-4\myarrowsize-.5\pgflinewidth} 
        \pgfarrowsrightextend{.7\pgflinewidth}
    }{
        \setlength{\myarrowsize}{0.4pt} 
        \addtolength{\myarrowsize}{.3\pgflinewidth}  
        \setlength{\myoldlinewidth}{\pgflinewidth}
        \pgfsetroundjoin
        \pgfsetlinewidth{0.0001pt}
        \pgfpathmoveto{\pgfpoint{0.43\myarrowsize}{0}}
        \pgfpatharc{0}{70}{0.14\myarrowsize}
        \pgfpatharc{-80}{-169.5}{4\myarrowsize}
        \pgfpatharc{150}{189}{0.95\myarrowsize and 0.95\myarrowsize}
        \pgfpatharc{0}{-40}{15\myarrowsize}
        \pgfpathmoveto{\pgfpoint{0.43\myarrowsize}{0}}
        \pgfpatharc{0}{-70}{0.14\myarrowsize}
        \pgfpatharc{80}{169.5}{4\myarrowsize}
        \pgfpatharc{-150}{-189}{0.95\myarrowsize and 0.95\myarrowsize}
        \pgfpatharc{0}{40}{15\myarrowsize}
        \pgfpathclose
        \pgfsetstrokeopacity{0.25}
        \pgfusepathqfillstroke
    }

    \pgfarrowsdeclare{myonto}{myonto}{
        \pgfsetdash{}{0pt} 
        \pgfsetbeveljoin 
        \pgfsetroundcap 
        \setlength{\myarrowsize}{0.5pt}
        \addtolength{\myarrowsize}{.5\pgflinewidth}
        \pgfarrowsleftextend{-4\myarrowsize-.5\pgflinewidth} 
        \pgfarrowsrightextend{.7\pgflinewidth}
    }{
        \setlength{\myarrowsize}{0.4pt} 
        \addtolength{\myarrowsize}{.3\pgflinewidth}  
        \setlength{\myoldlinewidth}{\pgflinewidth}
        \pgfsetroundjoin
        \pgfsetlinewidth{0.0001pt}
        \pgfpathmoveto{\pgfpoint{0.43\myarrowsize}{0}}
        \pgfpatharc{0}{70}{0.14\myarrowsize}
        \pgfpatharc{-80}{-169.5}{4\myarrowsize}
        \pgfpatharc{150}{189}{0.95\myarrowsize and 0.95\myarrowsize}
        \pgfpatharc{0}{-40}{15\myarrowsize}
        \pgfpathmoveto{\pgfpoint{-7\myarrowsize}{0}}
				\pgfpatharc{0}{70}{0.14\myarrowsize}
        \pgfpatharc{-80}{-169.5}{4\myarrowsize}
        \pgfpatharc{150}{189}{0.95\myarrowsize and 0.95\myarrowsize}
        \pgfpatharc{0}{-40}{15\myarrowsize}
        \pgfpathmoveto{\pgfpoint{0.43\myarrowsize}{0}}
        \pgfpatharc{0}{-70}{0.14\myarrowsize}
        \pgfpatharc{80}{169.5}{4\myarrowsize}
        \pgfpatharc{-150}{-189}{0.95\myarrowsize and 0.95\myarrowsize}
        \pgfpatharc{0}{40}{15\myarrowsize}
			  \pgfpathmoveto{\pgfpoint{-7\myarrowsize}{0}}
        \pgfpatharc{0}{-70}{0.14\myarrowsize}
        \pgfpatharc{80}{169.5}{4\myarrowsize}
        \pgfpatharc{-150}{-189}{0.95\myarrowsize and 0.95\myarrowsize}
        \pgfpatharc{0}{40}{15\myarrowsize}
        \pgfpathclose
        \pgfsetstrokeopacity{0.25}
        \pgfusepathqfillstroke
        \pgfusepathqfillstroke
    }

 \pgfarrowsdeclare{myhook}{myhook}{
        \setlength{\myarrowsize}{0.6pt}
        \addtolength{\myarrowsize}{.5\pgflinewidth}
        \pgfarrowsleftextend{-4\myarrowsize-.5\pgflinewidth} 
        \pgfarrowsrightextend{.7\pgflinewidth}
    }{
        \setlength{\myarrowsize}{0.6pt} 
        \addtolength{\myarrowsize}{.5\pgflinewidth}  
        \pgfsetdash{}{+0pt}
        \pgfsetroundcap
        \pgfpathmoveto{\pgfqpoint{-2pt}{-6\pgflinewidth}}
        \pgfpathcurveto
            {\pgfqpoint{4\pgflinewidth}{-4.667\pgflinewidth}}
            {\pgfqpoint{4\pgflinewidth}{0pt}}
            {\pgfpointorigin}
        \pgfusepathqstroke
    }

		\pgfarrowsdeclare{my to}{my to}
{
  \pgfarrowsleftextend{-2\pgflinewidth}
  \pgfarrowsrightextend{\pgflinewidth}
}
{
  \pgfsetlinewidth{0.8\pgflinewidth}
  \pgfsetdash{}{0pt}
  \pgfsetroundcap
  \pgfsetroundjoin
  \pgfpathmoveto{\pgfpoint{-5.5\pgflinewidth}{7.5\pgflinewidth}}
  \pgfpathcurveto
  {\pgfpoint{-4.0\pgflinewidth}{0.1\pgflinewidth}}
  {\pgfpoint{0pt}{0.25\pgflinewidth}}
  {\pgfpoint{0.75\pgflinewidth}{0pt}}
  \pgfpathcurveto
  {\pgfpoint{0pt}{-0.25\pgflinewidth}}
  {\pgfpoint{-4.0\pgflinewidth}{-0.1\pgflinewidth}}
  {\pgfpoint{-5.5\pgflinewidth}{-7.5\pgflinewidth}}
  \pgfusepathqstroke
}

		\pgfarrowsdeclare{mytodashed}{mytodashed}
{
 \pgfarrowsleftextend{-2\pgflinewidth}
 \pgfarrowsrightextend{\pgflinewidth}
}
{
 \pgfsetlinewidth{0.8\pgflinewidth}
  \pgfsetdash{{0.8cm}{0.8cm}}{0cm}
\pgfsetroundcap
\pgfsetroundjoin
  \pgfpathmoveto{\pgfpoint{-5.5\pgflinewidth}{7.5\pgflinewidth}}
\pgfpathcurveto
{\pgfpoint{-4.0\pgflinewidth}{0.1\pgflinewidth}}
 {\pgfpoint{0pt}{0.25\pgflinewidth}}
{\pgfpoint{0.75\pgflinewidth}{0pt}}
\pgfpathcurveto
{\pgfpoint{0pt}{-0.25\pgflinewidth}}
  {\pgfpoint{-4.0\pgflinewidth}{-0.1\pgflinewidth}}
 {\pgfpoint{-5.5\pgflinewidth}{-7.5\pgflinewidth}}
\pgfusepathqstroke
}

\tikzstyle{vecArrow} = [thick, decoration={markings,mark=at position
   1 with {\arrow[semithick]{open triangle 60}}},
   double distance=1.4pt, shorten >= 5.5pt,
   preaction = {decorate},
   postaction = {draw,line width=1.4pt, white,shorten >= 4.5pt}]
\tikzstyle{innerWhite} = [semithick, white,line width=1.4pt, shorten >= 4.5pt]

	\makeatletter
	\newcommand\POSITION[3]{%
	\begingroup
	\@tempdim@x=0cm
	\@tempdim@y=\paperheight
	\advance\@tempdim@x#1
	\advance\@tempdim@y-#2
	\put(\LenToUnit{\@tempdim@x},\LenToUnit{\@tempdim@y}){#3}%
	\endgroup
	}
\makeatother

\addto\captionsenglish{}

\begin{document}

	\begin{abstract}
	We define a tower of injections of  $\tilde{C}$-type Coxeter groups $W(\tilde C_{n})$ for $n\geq 1$. We define a tower of Hecke algebras and we use the faithfulness at the Coxeter level to show that this last tower is a tower of injections. Let $W^c(\tilde C_{n})$ be the set of  fully commutative elements in $W(\tilde C_{n})$, we classify  the elements of $W^c(\tilde C_{n})$  and  give a normal form for them.  We  
use this normal form to define two injections from $W^c(\tilde C_{n-1})$ into $W^c(\tilde C_{n})$. We then define the tower of affine Temperley-Lieb algebras of type $\tilde{C }$ and use the injections above to prove the faithfullness of this tower.    
	\end{abstract}

 	\maketitle
	

\section{Introduction}

In his fundamental paper \cite{J87}, in 1987, Jones has shown that the 
original Temperley-Lieb algebra, a certain finite-dimensional algebra appearing in   statistical mechanics as well as
 in physics, could be viewed as a quotient   of the Hecke algebra of type $A_n$ \cite[\S 11 and Note 13.20]{J87}. He also showed that the Markov trace, 
the famous invariant of braids, was actually a trace defined on the tower of Temperley-Lieb algebras of type $A$ --- the natural embedding of a Coxeter diagram of type $A_n$ into a   Coxeter diagram of type $A_{n+1}$, for $n \ge 1$, 
inducing monomorphisms  of the associated   Coxeter groups, Hecke algebras and Temperley-Lieb algebras. 

This opened a door to intensive study. In the same year (1995),  Fan \cite{Fan_1995}  
 (in the simply laced case) and Graham 
\cite{Graham}, in their theses, defined and studied generalized Temperley-Lieb algebras, that are quotients of Hecke algebras of Coxeter groups generalizing the original one, and identified a basis for them,  indexed by elements of the associated Coxeter group enjoying special properties, 
while Stembridge \cite{St96} defined and studied  {\it fully commutative elements} in a Coxeter group: 

{\it  
An element $w$ of a Coxeter group $W$ is {\rm fully commutative} 
if any reduced
expression for $w$ can be obtained from any other by means of braid relations that only involve commuting generators.} 
 
Since then, fully 
 commutative elements of Coxeter groups, that  do index a basis of the associated generalized Temperley-Lieb algebra, have been studied in their own right by numerous authors, as well as generalized Temperley-Lieb algebras. 
 Nevertheless the  case of affine Coxeter groups was not much studied until recently. 

An obvious difficulty 
of the affine case is that we have to deal with infinite groups and infinite-dimensional algebras. 
Yet a main difficulty of the affine case is the lack of parabolicity.  
In the series of finite Coxeter groups 
$(W(A_n))_{n \ge 1}$, $W(B_n)_{n \ge 2}$,  $W(D_n)_{n \ge 3}$,   the $n$-th  group is a parabolic subgroup of  the $(n+1)$-th one,  so that the associated    Hecke algebras and  Temperley-Lieb algebras inject naturally into one another (that is, the $n$-th into the $(n+1)$-th).  This is no longer the case for affine Coxeter groups, for which  defining morphisms and proving injectivity is not straightforward. 

For instance, the author in his thesis \cite{Sadek_Thesis} defined and studied 
a tower of affine Temperley-Lieb algebras of type $\tilde A$ and defined  on this tower the notion of a Markov trace, 
for which he proved existence and uniqueness, hence giving an invariant of affine links \cite{Sadek_2013_2}. A crucial tool in this study was to produce a normal form for fully commutative elements in Coxeter groups of type $\tilde A$. It is only later that he could prove the faithfulness  of this tower, by means of a faithful tower of 
 fully commutative elements   of type $\tilde A$ depending on their normal form \cite{Sadek_2015}.

\medskip

The center of interest of the present work is this "tower"-point of view on the structures of type $\tilde{C}$, namely Coxeter groups, Hecke algebras and Temperley-Lieb algebras, 
for which 
  fully commutative elements index at least two well-known bases. The heart of this work is the production of a normal form for 
  fully commutative elements in  Coxeter groups of type $\tilde{C}$, normal form that is subsequently used to build an injective tower  of 
 fully commutative elements   of type $\tilde C$ and ultimately prove the faithfulness of the tower 
of   Temperley-Lieb algebras of type $\tilde{C}$ that we define.   \\

In his thesis \cite{Ernst_2008} Ernst has given a faithful diagrammatic presentation for $\tilde{C}$-type Temperley-Lieb algebras  (see also \cite{Ernst_2012}). The method  there was to  classify fully commutative elements in Coxeter groups of type $\tilde{C}$ that are irreducible under ``weak star reduction''. In this paper we use simple algebraic methods, in particular  the notion of affine length (Definition \ref{AL}) to classify fully commutative elements by giving a normal form for each (Theorem \ref{FC}). 

This is the second affine normal form of fully commutative elements, after the one for  type $\tilde A$ in \cite{Sadek_2015}. The author has also obtained  such a normal form in types $\tilde B$ and $\tilde D$,
these normal forms allowed us in an unpublished work  to calculate the length generating function $ \sum_{w \in \tilde W^c } q^{l(w)}$ explicitly, i.e. as a rational polynomial in $q$ (see\cite{HJ}). We notice that affine fully commutative elements do not behave in a wild way in the four infinite families of affine Coxeter groups:  they have a subword that is a power of  a Coxeter element, except for the $\tilde C$ type  where  we have to distinguish  between two families of   elements, one of which involves as a subword a power of the element 
$[  -n,n ] \;  t_{n+1} = \sigma_{n}  \sigma_{n-1} \dots  \sigma_{1} t  \sigma_{1} \dots  \sigma_{n-1}  \sigma_{n} t_{n+1}$ (Theorem  \ref{FC}). For normal forms of fully commutative elements in the finite Coxeter groups of   types $ A, B$ and $ D$ we refer to  \cite{St}.  

Authors in \cite{CFC} have studied cyclically fully commutative elements.  
We    give in Remark \ref{CC} some examples of the way in which  our normal form can express such elements.    \\

Another motive for learning more about fully commutative elements of type  $\tilde C$ is to give an answer to Green's hypothesis, see  \cite{ Green2007}, especially Property B, that is the existence of a symmetric bilinear form with some nice properties linked to a known method given by Green (see for example \cite{Green2006}) about a Jones-like trace allowing us to compute the coefficients of  Kazhdan-Lusztig polynomials. In the  $\tilde A$  case for example the author has used the normal form to define a class of elements such that any trace is uniquely defined by its value on them \cite[Definition 4.5, Theorem 4.6]{Sadek_2013_2}, and that is what we intend to do later with the type  $\tilde C$: we wish  
  to define and classify such traces in the type $\tilde C$. The topological interest here comes from the fact that the closures of $\tilde C$-type braids are the  links in a double torus.  \\

A further   motivation is  linked to the "parabolic-like presentation" defined in \cite{Sadek_Thesis} in the $\tilde A$ case and recently by the author for  type $\tilde C$. We would like to prove that the behavior of our structures does not change much if we define them with another presentation,  where the tower would be defined by adding a generator, at the cost of replacing a  braid relation by a "braid-like" relation. \\

This paper is divided into three parts and  it is organized as follows:  

 The first part is centered around the towers of Coxeter groups and Hecke algebras of type $\tilde C$. In section 2, we define a group morphism from the Coxeter group $W(\tilde{C}_{n-1}) $ of type $\tilde{C}_{n-1}$ 
to $W(\tilde{C}_{n}   )$ and  we prove the injectivity of this morphism by considering  $W(\tilde{C}_{n-1} ) $   and $W(\tilde{C}_{n} ) $  as   subgroups  of $W(\tilde{A}_{2n-1} ) $ and  $W(\tilde{A}_{2n+1} ) $ respectively 
 (Corollary \ref{Coxeterinj}).  In section 3, we let $K$ be a commutative ring with identity and we let $q$ be an invertible element of $K$. We define the tower of  $\tilde C$-type Hecke algebras $H\tilde C_n(q)$. We  
prove the injectivity of this tower for $K =  \mathbf Q[q,q^{-1}]$  using the specialization  $q=1$  (Proposition \ref{pr_3_3_3}).  
 
 The second part, section 4,  is centered around the normal form of fully commutative elements  of $W (\tilde C_{n})$. After recalling  the normal form   in type $B$ given by Stembridge   in  \cite[Theorem 5.1]{St}, we define the affine length of an element of  $W (\tilde C_{n})$ and  we   establish the $\tilde C$-version of Stembridge's result, namely Theorem  \ref{FC}, that determines a normal form for fully commutative elements of $W (\tilde C_{n})$. This is the main result in this section and the base point of what follows.  
 
In the third part we describe two towers of fully commutative elements which will lead to the faithfulness of the tower of  $\tilde C$-type Temperley-Lieb algebras. In section 5, we define    two injections $I$ and $J$ from  the set $W^c(\tilde C_{n-1})$ of fully commutative elements in $W (\tilde C_{n-1})$ into  $W^c(\tilde C_{n})$ and their essential  properties, this is Theorem \ref{IJ},  of which the proof depends totally on the normal form. In section 6, we define  the tower of $\tilde C$-type Temperley-Lieb algebras, then, as an application of our normal form, we prove     the faithfulness of the arrows of this tower  in Theorem  \ref{R}. The proof uses in a crucial way the injections $I$ and $J$ of the previous section.  \\

\section{A faithful tower  of $\tilde C$-type Coxeter groups}\label{group tower}

Let $(W,S)$ be a Coxeter system with associated Coxeter diagram $\Gamma$. Let $w\in W(\Gamma)$ or simply $W$. We denote by $l(w)$ the length of a (any) reduced expression of $w$.   We define $\mathscr{L} (w) $ to be the set of $s\in S$ such that $l(sw)<l(w)$, in other terms  $s$ appears at the left edge of some reduced expression of $w$.  We define $\mathscr{R}(w)$ similarly, on the right.\\

Consider the $B$-type Coxeter group with $n+1$ generators $W(B_{n+1})$, with the following Coxeter diagram:\\
			
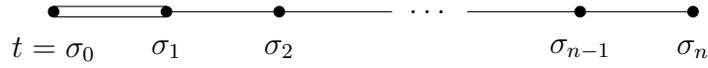
			\begin{figure}[ht]
				\centering
				\begin{tikzpicture}

 \filldraw (-1.5,0) circle (2pt);
  \node at (-1.5,-0.5) {$t=\sigma_0$}; 

  \draw (-1.5,-0.07) -- (0, -0.07);
  \draw (-1.5,0.07) -- (0, 0.07);
  
  \filldraw (0,0) circle (2pt);
  \node at (0,-0.5) {$\sigma_{1}$}; 
   
  \draw (0,0) -- (1.5, 0);

  \filldraw (1.5,0) circle (2pt);
  \node at (1.5,-0.5) {$\sigma_{2}$};

  \draw (1.5,0) -- (3, 0);

  \node at (3.5,0) {$\dots$};

  \draw (4,0) -- (5.5, 0);
  
  \filldraw (5.5,0) circle (2pt);
  \node at (5.5,-0.5) {$\sigma_{n-1}$};
 
  \draw (5.5,0) -- (7, 0);
  
  \filldraw (7,0) circle (2pt);
  \node at (7,-0.5) {$\sigma_{n}$};

               \end{tikzpicture}
			\end{figure}

Now let $W(\tilde{C}_{n+1}) $ be the affine Coxeter group of $\tilde{C}$-type with $n+2$ generators in which $W(B_{n+1})$ could be seen a parabolic subgroup in two ways. We make our choice by presenting $W(\tilde{C}_{n+1} ) $ with the following Coxeter diagram: \\
			
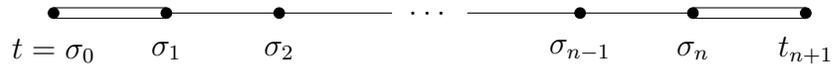
			\begin{figure}[ht]
				\centering
				\begin{tikzpicture}

 \filldraw (-1.5,0) circle (2pt);
  \node at (-1.5,-0.5) {$t=\sigma_0$}; 

  \draw (-1.5,-0.07) -- (0, -0.07);
  \draw (-1.5,0.07) -- (0, 0.07);
  
  \filldraw (0,0) circle (2pt);
  \node at (0,-0.5) {$\sigma_{1}$}; 
   
  \draw (0,0) -- (1.5, 0);

  \filldraw (1.5,0) circle (2pt);
  \node at (1.5,-0.5) {$\sigma_{2}$};

  \draw (1.5,0) -- (3, 0);

  \node at (3.5,0) {$\dots$};

  \draw (4,0) -- (5.5, 0);
  
  \filldraw (5.5,0) circle (2pt);
  \node at (5.5,-0.5) {$\sigma_{n-1}$};
 
  \draw (5.5,0) -- (7, 0);
  
  \filldraw (7,0) circle (2pt);
  \node at (7,-0.5) {$\sigma_{n}$};

   \draw (7,-0.07) -- (8.5, -0.07);
  \draw (7,0.07) -- (8.5, 0.07);
  
   \filldraw (8.5,0) circle (2pt);
  \node at (8.5,-0.5) {$t_{n+1}$};

               \end{tikzpicture}
			\end{figure}
In other words the group  $W(\tilde{C}_{n+1} ) $ has a presentation given by the set of generators 
$\{ \sigma_{0}, \sigma_{1}, \dots, \sigma_{n},t_{n+1} \}$ and the relations: 
$$ \begin{aligned}
&  t_{n+1}^2 = 1 \text{ and } \sigma_i^2 = 1  \text{ for } 0\le i \le n ; \\ 
&\sigma_i \sigma_{j} = \sigma_{j} \sigma_i \text{ for } 0\le i, j \le n, \ |i-j|\ge  2 ; \\
&\sigma_i t_{n+1} = t_{n+1}\sigma_i \text{ for } 0\le i < n ; \\ 
&\sigma_i \sigma_{i+1} \sigma_i = \sigma_{i+1} \sigma_i\sigma_{i+1}  \text{ for } 1\le i \le n-1;\\ 
&\sigma_0 \sigma_{1}\sigma_0 \sigma_{1} =  \sigma_{1}\sigma_0 \sigma_{1}\sigma_0  ; \\
&\sigma_n t_{n+1}\sigma_n t_{n+1}= t_{n+1}\sigma_n t_{n+1}\sigma_n. 
\end{aligned}
$$

It is easy to check that the subset $\{ \sigma_{0}, \sigma_{1}, \dots, \sigma_{n-1},\sigma_n t_{n+1}\sigma_n \}$ 
of $W(\tilde{C}_{n+1} ) $ satisfies the defining relations for $W(\tilde{C}_{n} ) $. 
We may thus define the tower of   $\tilde C$-type Coxeter groups by defining the following   group homomorphism,  for $ n \ge 2  $: \\ 

\begin{eqnarray}
					  P_n: W(\tilde C_{n }) &\longrightarrow& W(\tilde C_{n+1})\nonumber\\
					\sigma_{i} &\longmapsto& \sigma_{i}$ ~~~ \text{for} $0\leq i\leq n-1 \nonumber\\
					t_{n} &\longmapsto& \sigma_n t_{n+1} \sigma_n \nonumber\\
  \nonumber
				\end{eqnarray}

		The  goal of this section is to prove the faithfulness of this arrow.	We will do so by  embedding our $\tilde C$-type Coxeter groups in $\tilde A$-type Coxeter groups and using the faithfulness of a relevant arrow between $\tilde A$-type Coxeter groups, as follows. 
Let $W(\tilde A_{n-1})$ be the $\tilde A$-type  Coxeter group with $n$ generators, say  
$\left\{ s_{1},  \dots ,s_{n-1},  a_{n}    \right\} $, with the following Coxeter diagram: 

			
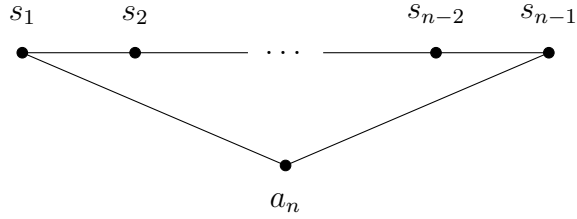
\begin{figure}[ht]
				\centering
				\begin{tikzpicture}

 \node at (0,0.5) {$s_{1}$}; 
  \filldraw (0,0) circle (2pt);
   
  \draw (0,0) -- (1.5, 0);
  
  \node at (1.5,0.5) {$s_{2}$};
  \filldraw (1.5,0) circle (2pt);

\draw (1.5,0) -- (3, 0);

  \node at (3.5,0) {$\dots$};

  \draw (4,0) -- (5.5, 0);

  \node at (5.5,0.5) {$s_{n-2}$};
  \filldraw (5.5,0) circle (2pt);
 
  \draw (5.5,0) -- (7, 0);
  
  \node at (7,0.5) {$s_{n-1}$};
  \filldraw (7,0) circle (2pt);

  \draw (7,0) -- (3.5, -1.5);
  
  \filldraw (3.5, -1.5) circle (2pt);
  \node at (3.5, -2) {$a_{n}$};

  \draw (3.5, -1.5) -- (0, 0);
               \end{tikzpicture}
			\end{figure}
 
\medskip
The group 
 $W(\tilde C_{n}) $ can be seen as the group of fixed points in $W(\tilde A_{2n-1})$ by some involution, 
so that we have 
  an embedding of $W(\tilde C_{n}) $  in $W(\tilde A_{2n-1})$ given as follows (see \cite[\S 4, \S 5]{Digne_2012}\cite[Corollaire 3.5]{Hee}; we compose here the embedding given in   \cite{Digne_2012} with the Dynkin 
automorphism 
of $W(\tilde C_{n}) $) to make it more convenient for our purpose):     \\

     \begin{eqnarray}
                                            i_n: W(\tilde C_{n}) &\longrightarrow& W(\tilde A_{2n-1})\nonumber\\
					\sigma_{i} &\longmapsto& s_{n-i} s_{n+i}$ ~~~ \text{for} $1\leq i\leq n-1 \nonumber\\
					                 t &\longmapsto& s_{n}  \nonumber\\
                                                        t_{n} &\longmapsto&  a_{2n}   \nonumber 
                                                       \end{eqnarray}
       \begin{figure}[ht]
				\centering
				\begin{tikzpicture}

 \filldraw (-1.5,0) circle (2pt);
  \node at (-1.5,-0.5) {$t$}; 

  \draw (-1.5,-0.07) -- (0, -0.07);
  \draw (-1.5,0.07) -- (0, 0.07);
  
  \filldraw (0,0) circle (2pt);
  \node at (0,-0.5) {$\sigma_{1}$}; 
   
  \draw (0,0) -- (1.5, 0);

  \filldraw (1.5,0) circle (2pt);
  \node at (1.5,-0.5) {$\sigma_{2}$};

  \draw (1.5,0) -- (3, 0);

  \node at (3.5,0) {$\dots$};

  \draw (4,0) -- (5.5, 0);
  
  \filldraw (5.5,0) circle (2pt);
  \node at (5.5,-0.5) {$\sigma_{n-1}$};
 
  \draw (5.5,-0.07) -- (7, -0.07);
  \draw (5.5,0.07) -- (7, 0.07);
  
   \filldraw (7,0) circle (2pt);
  \node at (7,-0.5) {$t_{n}$};

 \filldraw (-1.5,-3) circle (2pt);
  \node at (-2.2,-3) {$s_{n}$}; 

  \draw (-1.5,-3) -- (0, -2);
  
  \filldraw (0,-2) circle (2pt);
  \node at (0,-1.5) {$s_{n-1}$}; 
   
  \draw (0,-2) -- (1.5, -2);

  \filldraw (1.5,-2) circle (2pt);
  \node at (1.5,-1.5) {$s_{n-2}$};

  \draw (1.5,-2) -- (3, -2);

  \node at (3.5,-2) {$\dots$};

  \draw (4,-2) -- (5.5, -2);
  
  \filldraw (5.5,-2) circle (2pt);
  \node at (5.5,-1.5) {$s_1$};
 
\draw (5.5,-2) -- (7, -3);
     
   \filldraw (7,-3) circle (2pt);
  \node at (7.7,-3) {$a_{2n}$};

  \draw (-1.5,-3) -- (0, -4);
  
   \filldraw (0,-4) circle (2pt);
  \node at (0,-4.5) {$s_{n+1}$}; 
   
  \draw (0,-4) -- (1.5, -4);

  \filldraw (1.5,-4) circle (2pt);
  \node at (1.5,-4.5) {$s_{n+2}$};

  \draw (1.5,-4) -- (3, -4);

  \node at (3.5,-4) {$\dots$};

  \draw (4,-4) -- (5.5, -4);
  
  \filldraw (5.5,-4) circle (2pt);
  \node at (5.5,-4.5) {$s_{2n-1}$};
 
\draw (5.5,-4) -- (7, -3);

               \end{tikzpicture}
			\end{figure}
 
We now recall from   \cite{Sadek_Thesis} and  \cite[Lemma 4.1]{Sadek_2015}
  the following monomorphism: 
	\begin{eqnarray}
					  I_n: W(\tilde A_{n-1}) &\longrightarrow& W(\tilde A_{n})\nonumber\\
					s_{i} &\longmapsto& s_{i}$ ~~~ \text{for} $1\leq i\leq n-1 \nonumber\\
					a_{n} &\longmapsto& s_{n} a_{n+1}s_{n}   \nonumber 
				\end{eqnarray}
		Letting $\phi_{2n+1}$ be 
			  the Coxeter automorphism   of  $W(\tilde A_{2n+1})$ given by $$s_1 \mapsto s_2  \mapsto \dots s_{2n}  \mapsto s_{2n+1}  \mapsto a_{2n+2}  \mapsto s_{1} ,$$  the composition   $L_n=\phi_{2n+1}   I_{2n+1}   I_{2n }$ is a  monomorphism.

     \begin{lemma} 
     
     The following diagram commutes  for any $n > 1$: 

\begin{center}    

	\begin{tikzpicture}

			\matrix[matrix of math nodes,row sep=1cm,column sep=1cm]{
			|(A)| W(\tilde A_{2n-1})    & & & &    |(B)| W(\tilde A_{2n+1})  \\
			                              & & & &                      \\								
			|(C)|   W(\tilde{C}_{n} )             & & & &    |(D)|   W(\tilde{C}_{n+1} )      \\
				};

				\path (A) edge[-myhook,line width=0.42pt]  node[above, xshift=-5mm, yshift=-2mm, rotate=0] {\footnotesize $i_{n} $}    (C);
\path (C) edge[-myto,line width=0.42pt]      (A);
				
					\path (B) edge[-myhook,line width=0.42pt]  node[above, xshift=1.5mm, yshift=0mm, rotate=0] {\footnotesize $L_{n}$}      (A);
\path (A)  edge[-myto,line width=0.42pt]    (B);

				\path (C) edge[-myto,line width=0.42pt]   node[below, xshift=1.5mm, yshift=0mm, rotate=0]    {\footnotesize $P_{n}$}      (D);
								
				 \path (B) edge[-myhook,line width=0.42pt]  node[above, xshift=-5mm, yshift=-2mm, rotate=0] {\footnotesize $i_{n+1} $}    (D);
				\path (D) edge[-myto,line width=0.42pt]     (B);

		\end{tikzpicture}	
	\end{center}
     \end{lemma}                                        
     
     \begin{proof}
        For $1\leq i\leq n-1$ we have 
 $$L_n  i_n(\sigma_{i}) = L_n (s_{n-i} s_{n+i}) = s_{n-i+1} s_{n+i+1} = s_{n+1-i} s_{n+1+i} = i_{n+1} ( \sigma_i) =i_{n+1}   P_n  ( \sigma_{i}) , $$        
      while for $i=0$:   $L_n  i_n(t) =  L_n (s_{n} )= s_{n+1}$ which is exactly $i_{n+1} (t)$, that is $ i_{n+1}  P_n (t)$.
      
      Now consider  $L_n  i_n(t_{n})= L_n(a_{2n})$. We have: 
      $$ 
\begin{aligned}
L_n(a_{2n}) &= 
       \phi_{2n+1} I_{2n+1} (s_{2n}  a_{2n+1}  s_{2n}) 
=   \phi_{2n+1} 
(s_{2n}s_{2n+1} a_{2n+2} s_{2n+1}s_{2n}) \\  &= 
        s_{2n+1}a_{2n+2}s_{1}a_{2n+2}s_{2n+1}=s_{2n+1}s_{1} a_{2n+2} s_{1}s_{2n+1}, 
\end{aligned}$$
 using the braid relation between    $s_{1}$ and $ a_{2n+2}$.      On the other hand we have: 
     $$ i_{n+1}  P_n  (t_n) = i_{n+1} (  \sigma_n  t_{n+1}    \sigma_n ) = s_{n+1-n}s_{n+1+n} a_{2n+2} s_{n+1-n} s_{n+1+n},$$
      which is $s_{2n+1}s_{1} a_{2n+2} s_{1}s_{2n+1}$, and the commutation of the diagram follows.                                        
     \end{proof}

    \begin{corollary}\label{Coxeterinj}  $P_n   : W(\tilde C_{n })  \longrightarrow  W(\tilde C_{n+1})$ is an  injection for any $n > 1$.
 \end{corollary}

\section{The tower of  $\tilde C$-type Hecke algebras}\label{HC}      

Let for the moment $K$ be an arbitrary commutative ring with identity; 
 we mean by algebra in what follows $K$-algebra. We recall \cite[Ch. IV \S 2 Ex. 23]{Bourbaki_1981} that for  a given Coxeter graph $\Gamma$ and a corresponding Coxeter system $(W,S)$,   there is a unique algebra structure on the free $K$-module with basis $ \left\{  g_w \ | \ w \in W(\Gamma) \right\} $ satisfying, for a given $q \in K$: 
\begin{equation*}\label{definingrelations} 	
	 \begin{aligned}
		 &g_{s} g_{w} =g_{sw}     ~~~~~~~~~~~~~~~~~~~~  \text{ for } s \notin \mathscr{L} (w) , \\
		 &g_{s} g_{w} =qg_{sw}+ (q-1)g_w ~~  \text{ for } s \in \mathscr{L} (w). \\
			  \end{aligned}    \qquad 
		\end{equation*}
We denote this algebra by $H\Gamma(q)$  and call it the the $\Gamma$-type Hecke algebra.   This algebra has a presentation ({\it loc.cit.}) given by generators $ \left\{  g_s \ | \   s \in S \right\} $ and relations 
$$
	 \begin{aligned}
		  g_{s}^2   &=q + (q-1)g_s ~~~~~~~~ \   \text {for } s \in S,  \\
		  (g_{s} g_{t})^r &=(g_{t} g_{s})^r    ~~~~~~~~~~~~~~~~~   \! \text{ for } s, t \in S 
\text{ such that } st \text{ has order }  2r , \\
  (g_{s} g_{t})^r g_s&=(g_{t} g_{s})^r  g_t   ~~~~~~~~~~~~~~  \text{ for } s, t \in S 
\text{ such that } st \text{ has order }  2r+1 .  
			  \end{aligned}   
$$
We   assume in what follows that $q$ is invertible in $K$. In this case  the first defining relation above implies that $ g_{s}$, for $s \in S$, is invertible 
with inverse 
\begin{equation}\label{inverse} 
	g^{-1}_s = 		 \frac{1}{q} \;  	g_s +  \frac{q-1}{q}.
 \end{equation}

		We consider the  $\tilde{C_n}$-type (resp. $B_n$-type)   Hecke algebra  $H\tilde{ C}_{n} (q) $ (resp. $HB_{n}(q)$) corresponding to the Coxeter group  $W(\tilde{C}_{n} )$ (resp. $W(B_{n})$), for $n\ge 2$. 
Regarding $W(B_{n})$ as a parabolic subgroup of $W(\tilde{C}_{n} )$ as in the previous paragraph, we view 
  $HB_n(q) $ as the  subalgebra of  $H\tilde{C}_{n} (q) $ generated by $\{g_{\sigma_{0}},g_{\sigma_{1}} , \dots g_{\sigma_{n-1}} \} $.  

Since $W(B_{n})$ is a parabolic subgroup of $W(B_{n+1})$ we can also see $HB_n(q) $ as a subalgebra 
of $HB_{n+1}(q)$, we thus  have the following  tower of   Hecke algebras: 	
\begin{eqnarray}
				K  ~\subset HB_2(q) ~~\cdots \subset HB_n(q) ~~\subset  HB_{n+1}(q) ~~\subset \cdots\nonumber 
							\end{eqnarray}

	The aim of this section is to define a similar tower of  $\tilde{C}$-type Hecke algebras, despite the fact that 
$W(\tilde{C}_{n} )$ is not a parabolic subgroup of $W(\tilde{C}_{n+1} )$. Let us write 
$\{e_{\sigma_{0}},  \dots ,e_{\sigma_{n-1}}, e_{t_n} \} $ for the generators of  $H\tilde{C}_n (q)$ and 
$\{g_{\sigma_{0}} , \dots , g_{\sigma_{n-1}}, g_{\sigma_{n }},g_{t_{n+1}} \} $ for those of 
$H\tilde{C}_{n+1} (q)$.   It is easy to check that 
$\{g_{\sigma_{0}} , \dots , g_{\sigma_{n-1}}, g_{\sigma_n} g_{t_{n+1}}g_{\sigma_n}^{-1}\} $ satisfies the defining relations for $H\tilde{C}_n (q)$,   we thus get the following morphism of algebras:

		\begin{equation}\label{defFn}
\begin{aligned}
					  R_n: H\tilde{C}_n (q)  &\longrightarrow   H\tilde{C}_{n+1} (q) \\
					e_{\sigma_i} &\longmapsto  g_{\sigma_i}  ~~~ ~~~ ~~~ \text{for }  0\leq i\leq n-1  \\
					e_{t_n} &\longmapsto  g_{\sigma_n} g_{t_{n+1}}g_{\sigma_n}^{-1}  .
\end{aligned}		
\end{equation}	 
 
\medskip
	
On the other hand, the group injection   $P_n:    W(\tilde C_{n })  \longrightarrow  W(\tilde C_{n+1})$ of Corollary \ref{Coxeterinj} extends to the group algebras, providing the following algebra monomorphism:
	$$\begin{aligned}
				P_n :K[W(\tilde{C}_{n } )] &\longrightarrow  K[W (\tilde{C}_{n+1} )]    \\
				\sigma_{i}&\longmapsto  \sigma_{i} \text{  for }    0\leq i\leq n-1   \\
				t_{n}&\longmapsto  \sigma_{n} t_{n+1} \sigma_{n}.    
\end{aligned}			$$

We let now  $K $   be the ring   $\mathbf Q[q,q^{-1}]$ of Laurent polynomials with rational coefficients and will prove:  
 		
	\begin{proposition}\label{diagcomm}
				The following diagram,   where $M_{n}$ and $M_{n+1}$ are the maps coming from specializing $q$ to 1, is commutative.	
	\begin{center}    

		\begin{tikzpicture}\label{FnIn}

			\matrix[matrix of math nodes,row sep=1cm,column sep=1cm]{
			|(A)|  H\tilde{C}_{n} (q)    & & & &    |(B)|  H\tilde{C}_{n+1} (q)  \\
			                              & & & &                      \\								
			|(C)| K[W(\tilde{C}_{n} )] )             & & & &    |(D)| K[W(\tilde{C}_{n+1} )]     \\
				};

				\path (A) edge[-myto,line width=0.42pt]  node[above, xshift=-5mm, yshift=-2mm, rotate=0] {\footnotesize $M_{n} $}    (C);
				
				 				\path (A) edge[-myto,line width=0.42pt] node[above, xshift=1.5mm, yshift=0mm, rotate=0] {\footnotesize $R_{n}$}      (B);
				
				\path (D) edge[-myhook,line width=0.42pt] node[below, xshift=1.5mm, yshift=0mm, rotate=0] {\footnotesize $P_n $}    (C);
				\path (C) edge[-myto,line width=0.42pt]      (D);

				\path (B) edge[-myto,line width=0.42pt]   node[above, xshift=5mm, , yshift=-2mm, rotate=0] {\footnotesize $M_{n+1}$}   (D);

		\end{tikzpicture}	

	\end{center}	
		\end{proposition}

			\begin{proof}	We will first prove the following Lemma, in which we simplify the notation by setting 
$R_n=R$ and $P_n=P$:

		\begin{lemma} \label{3_3_1}
				Let $w$ be any element in $W(\tilde{C}_{n } )$. Then:
				\begin{eqnarray}
					R(e_{w}) =Ag_{P(w)}+\frac{q-1}{q^r}\sum\limits_{x\in W(\tilde{C}_{n+1})} \lambda_{x}g_{x}, \nonumber
				\end{eqnarray}
				where $A $ belongs to $ q^\mathds Z$, the $ \lambda_{x}$ are polynomials in $q$  over $\mathds{Q} $ and $r  $ is a non-negative integer. 
				
			\end{lemma}

			\begin{proof}
				Suppose $l(w)=1$. If $w=e_{\sigma_i} $  for $0\leq i\leq n-1$, then 
			$
					R(e_{w})=R(e_{\sigma_i})=g_{\sigma_i} 
		$, 
while for $w= e_{t_n}$ we use  (\ref{inverse}) and get:				 
				\begin{eqnarray}
				 R(e_{t_n}) = g_{\sigma_{n}} g_{t_{n+1}} g_{\sigma_{n}}^{-1} = \frac{1}{q} g_{\sigma_{n}} g_{t_{n+1}} g_{\sigma_{n}} + \frac{1-q}{q} g_{\sigma_{n}} g_{t_{n+1}}
					= \frac{1}{q} g_{\sigma_{n} t_{n+1} \sigma_{n}} + \frac{1-q}{q} g_{\sigma_{n} t_{n+1}} \nonumber
				\end{eqnarray}
	as announced. 
			  
				Now take $w$  with $  l(w)\ge 2$ and suppose that the statement is true for any element of length $h$ where $h < l(w)$. If $w\in W(B_{n})$, then $P(w)=w $ and $R(e_{w})=g_{w}$,  hence our statement holds. Otherwise  $w$ can be written as $w=ut_{n}v$ where $v \in   W(B_{n})$  and   $l(w)=l(u)+l(v)+1$ and we have:
				\begin{eqnarray}
					R(e_{w})	=	R(e_{u})	R(e_{t_{n}})	R(e_{v})	=	R(e_{u})	g_{\sigma_{n}}	g_{t_{n+1}}	g^{-1}_{\sigma_{n}}	g_{v}. \nonumber
				\end{eqnarray}
Using (\ref{inverse}) and the induction hypothesis for  $e_{u}$ we get 
 $A \in  q^\mathds Z$,  polynomials  $\mu_{y}$ ($y \in W(\tilde{C}_{n+1})$)   in $\mathds{Q}[q]$ and $r  \in \mathds{Z}^+$ such that: 
\begin{eqnarray}
					R(e_{w})&=& \frac{1}{q} \bigg(Ag_{P(u)}+\frac{q-1}{q^r}\sum\limits_{y\in W(\tilde{C}_{n+1})} \mu_{y}g_{y}\bigg)g_{\sigma_{n}}g_{t_{n+1}}g_{\sigma_{n}}g_{v}   \nonumber\\ 
					& & \qquad  + \frac{q-1}{q}\bigg(Ag_{P(u)}+\frac{q-1}{q^r}\sum\limits_{y\in W(\tilde{C}_{n+1})}\mu_{y}g_{y}\bigg)g_{\sigma_{n}}g_{t_{n+1}}g_{v} \nonumber\\ 
&=& \frac{1}{q} Ag_{P(u)}g_{\sigma_{n}}g_{t_{n+1}}g_{\sigma_{n}}g_{v} + \frac{q-1}{q^{r'}}\sum\limits_{y\in W(\tilde{C}_{n+1})}\mu'_{y}g_{y} .\nonumber
				\end{eqnarray}
  for some   $\mu'_{y}$ ($y \in W(\tilde{C}_{n+1})$)   in $\mathds{Q}[q]$, finitely many non zero,  and some $r'\in \mathds{Z}^+$.   				
			 
				We can write  $g_{P(u)}g_{\sigma_{n}}=a(q-1)g_{P(u)}+b g_{P(u)\sigma_{n}}$, with 
$(a,b)= (0,1)$ or $(a,b)= (1,q)$.  Then we keep computing: 
				\begin{eqnarray}
					g_{P(u)}g_{\sigma_{n}}g_{t_{n+1}}&=&a(q-1)g_{P(u)}g_{t_{n+1}}+b g_{P(u)\sigma_{n}}g_{t_{n+1}} \nonumber\\
					&=& a(q-1) g_{P(u)} g_{t_{n+1}} +b \big(a'(q-1)g_{P(u)\sigma_{n}}+b' g_{P(u)\sigma_{n}t_{n+1}}\big)\nonumber\\
					&=&(q-1)\big(ag_{P(u)}g_{t_{n+1}}+a'bg_{P(u)\sigma_{n}}\big)+ bb'g_{P(u)\sigma_{n}t_{n+1}}  \nonumber
				\end{eqnarray}
with again $(a',b')= (0,1)$ or $  (1,q)$. In the same way we multiply on the right   by $g_{\sigma_{n}}$  and we see that there exist suitable polynomials $\mu''_{y}$   in $\mathds{Q}[q]$ such that: 

				\begin{eqnarray}
					g_{P(u)} g_{\sigma_{n}} g_{t_{n+1}} g_{\sigma_{n}} = (q-1) ( \sum\limits_{z\in W(\tilde{C}_{n+1})}\mu''_{z}g_{z} ) + q^{s} g_{P(u)\sigma_{ n } t_{n+1}\sigma_{n}} \nonumber
				\end{eqnarray} 				
where $0 \leq s \leq 3 $. But $ g_{P(u)\sigma_{n}t_{n+1}\sigma_{n}}=g_{P(u)P(t_{n})}=g_{P(ut_{n})}$.	 Moreover, since $l(w) = l(u)+l(t_{n})+l(v)$ we see directly that  $g_{P(u)\sigma_{n}t_{n+1}\sigma_{n}} g_v= 	g_{P(u\sigma_{n})v}=g_{P(u\sigma_{n}v)}= g_{P(w)}$. Finally: 								
	$$ \begin{aligned}
					R(e_{w})&= \frac{1}{q} A \left[(q-1)( \sum\limits_{z\in W(\tilde{C}_{n+1})}\mu''_{z}g_{z} g_v)+ q^{s}g_{P(w)}\right] + \frac{q-1}{q^{r'}}\sum\limits_{y\in W(\tilde{C}_{n+1})}\mu'_{y}g_{y}  \\ 
				&=	  Aq^{s-1}g_{P(w)}+(q-1)  \left[\frac{A}{q} ( \sum\limits_{z\in W(\tilde{C}_{n+1})}\mu''_{z}g_{z} g_v ) + \frac{1}{q^{r'}}\sum\limits_{y\in W(\tilde{C}_{n+1})}\mu'_{y}g_{y}\right] .\nonumber
				\end{aligned}	$$		
		The lemma follows.
		\end{proof}	
		
	We go back to the proof of the Proposition. 		
				The diagram commutes  if and only if, for each $w$ in $W(\tilde{C}_{n} )$, we have:
				\begin{eqnarray}
					P_n\big(M_{n}(e_{w})\big)=M_{n+1}\big(R_n(e_{w})\big). \nonumber
				\end{eqnarray}				
We have $P_n\big(M_{n}(e_{w})\big)=P_n(w)$  while, by specializing Lemma \ref{3_3_1} at $q=1$, we see that $M_{n+1}(R_n(e_{w}))$ is equal to 
				 $A (1) P_n(w)  $, that is,  $P_n(w)$, whence the result.  
						\end{proof}

\begin{proposition}\label{pr_3_3_3}	 
	Let $K= \mathbf Q[q,q^{-1}]$. 			The homomorphism   of algebras  $$R_n:  H\tilde{C}_n (q)   \longrightarrow  H\tilde{C}_{n+1} (q)$$ defined in (\ref{defFn}) is an injection. 		
			\end{proposition}	
	
			\begin{proof}		
				We will make use of the fact that the   diagram  in Proposition \ref{diagcomm} commutes and prove that the images of the basis elements of $H\tilde{C}_{n } (q)$ are linearly independent in $H\tilde{C}_{n+1} (q)$. 		
				Suppose  that there exists a finite subset $Z$ of $W(\tilde{C}_{n }) $ and non-zero  polynomials  $\lambda_{w}$, $w \in Z$,  in $\mathds{Q}[q,q^{-1}]$,  with $\sum\limits_{w\in Z }\lambda_{w}R_n(e_{w})=0$. 
				We can as well assume that the  $\lambda_{w}$ are in 
				$\mathds{Q}[q  ]$ (by multiplying by some power of $q$) and that they have no common factor (by factoring out common factors if any).  We do so.

	Since the diagram in Proposition \ref{diagcomm} commutes, we have: 
				\begin{eqnarray}
					M_{n+1}\big(R_n(e_{w})\big)=P_n\big(M_{n}(e_{w})\big)=P_n(w) .\nonumber
				\end{eqnarray}	
				We now apply $M_{n+1}$ to the dependence relation  $\sum\limits_{w\in Z }\lambda_{w}R_n(e_{w})=0$ to get:
$$
					\sum\limits_{w\in Z}\lambda_{w}(1)P_n(w)=P_n(\sum\limits_{w\in Z}\lambda_{w}(1)w)=0, $$  
	thus $\sum\limits_{w\in Z}\lambda_{w}(1)w=0$, 			which implies that  $\lambda_{w}(1)=0$  for every $w \in Z$.  	
				This last fact means that the polynomial $(q-1)$ divides every $\lambda_{w}$,   which contradicts our hypothesis, hence $R_n$ is an injection.			 	\end{proof}
			
			\begin{remark}\label{3_3_4} {\rm 
		Whenever we know that  $R_{n}$ is injective, we don't need anymore to distinguish between generators of 
$H\tilde{C}_{n } (q)$  and  $H\tilde{C}_{n+1} (q)$: we 
can denote the generators of $H\tilde{C}_{n } (q)$ by $g_t,g_{\sigma_{1}}, \dots ,g_{\sigma_{n-1}},g_{t_{n}} $ and the generators of $H\tilde{C}_{n+1} (q)$ by  $g_ t, g_{\sigma_{1}}, \dots, g_{\sigma_{n}},g_{t_{n+1}} $. }
		\end{remark}	
		
\section{A normal form for $\tilde C$-type fully commutative elements}\label{notations}

				In a given Coxeter group $W(\Gamma)$, we know that from a given reduced expression of an element $w$ we can arrive to any other reduced expression of $w$ only by applying braid relations \cite[\S 1.5 Proposition 5]{Bourbaki_1981}.  Among these relations there are commutation relations: those 
				that  correspond  to   generators $t$ and $s$ such that $st$ has order $2$.  \\

	             \begin{definition}
			Elements for which one can pass from any reduced expression to any other one only by applying commutation relations are called {\rm fully commutative elements}. We denote  by $W^{c}(\Gamma)$, or simply $W^{c}$,  the set of fully commutative elements in $W= W(\Gamma)$. \\
		    \end{definition}

We consider the set  $W^{c}(B_{n+1})$  of  fully commutative elements in $W(B_{n+1})$  and recall 
the description given by Stembridge in \cite{St}. Using the notation there and    the convention $t= \sigma_0$ we let:  		
$$ \begin{aligned}  \      [ i,j ]    
&= \sigma_i \sigma_{i+1} \dots \sigma_j   \    \text{ for } 0\le i\le j \le n \    \text{ and }  \    [ n+1, n]   = 1, 
\\
[ - i,j ]  &= \sigma_i \sigma_{i-1} \dots  \sigma_1 t\sigma_{1}  \dots \sigma_{j-1}  \sigma_j  \    \text{ for } 1\le i\le j \le n\    \text{ and }  \    [ 0, -1]   = 1 .
\end{aligned}  
$$

\begin{theorem}\label{1_2}{\rm \cite[Theorem 5.1]{St}}
    $W^c(B_{n+1})$ is the set of elements of the following form: 
 \begin{equation}\label{Stembridge}
[ l_1, g_1 ] 	[  l_2, g_2 ]  \dots  [l_r , g_r ]	
 \end{equation}
with $n\ge g_1 > \dots > g_r \ge 0$ and 
$|l_t| \le g_t$ for $1\le t \le r$, such that  either

\begin{enumerate}
\item $ l_1 > \dots > l_s  > l_{s+1} = \dots = l_r = 0  $ for some  $s \le r$, or 
\item $ l_1 > \dots > l_{r-1}  > -l_r > 0 $. \\
\end{enumerate}
 
 
\end{theorem} 
	For example:  $W^{c}(B_{2}) =   \left\{ 1, t, \sigma_{1}, t \sigma_{1}, \sigma_{1} t ,   t\sigma_{1}t, \sigma_{1} t \sigma_{1}    \right\} $.    We remark that if $r>1$, then $[ l_2, g_2 ] \dots [ l_s, g_s ]$ in (\ref{Stembridge})  above belongs to $W^c( B_ {n})$. 
We also notice that  if $ \sigma_{n} $ appears in form (\ref{Stembridge}) above, then  either it appears   only once  and 
we have $n = g_1 \ne -l_1$, 
 	or it appears    exactly twice and we have $n = g_1 =-l_1$.  \\

 	\begin{definition}
				An element $u$ in $W^{c}(B_{n+1})$ is called 
				{\rm extremal}  if   $ \sigma_{n} $  appears in a (any) reduced expression for $u$. In this case $u$ can be written in  one of the two following   forms: 	
 $$\begin{aligned}   
&\text{either } \quad 				 [-n, n] 
	\\				 
	&\text{or } \quad 			 [ l_1,  n ][ l_2, g_2 ] \dots [ l_s, g_s ]   
 \end{aligned} $$ 
with   $n > g_2 > \dots > g_s \ge 0$, 
$|l_t| \le g_t$ for $1\le t \le s$, $l_1 \ne -n$  and one of the conditions (1) and (2) of Theorem  \ref{1_2}.
			\end{definition}

In the group $W(\tilde{C}_{n+1}) $, the only braid relation involving $t_{n+1}$ (apart from commutation relations) is 
$$t_{n+1}\sigma_n t_{n+1}\sigma_n = \sigma_n t_{n+1}\sigma_n t_{n+1}$$  where  the number of occurrences of 
$t_{n+1}$ is the same on both sides. It follows (recall \cite[\S 1.5 Proposition 5]{Bourbaki_1981}) that  the number of times $t_{n+1}$ occurs in a 
 reduced expression of an element of $ W(\tilde{C}_{n+1}) $ does not depend of this reduced expression.    

				\begin{definition} \label{AL} Let $u \in W (\tilde C_{n+1})$. 
				We define the {\rm affine length}  of $u$  to be the number of times  $t_{n+1}$ occurs   in a (any) 
 reduced expression of $u$. We denote  it by $L(u)$. \\
			\end{definition}		
    
\begin{lemma}\label{lemmafull}
Let $w$ be  a fully commutative element in $W(\tilde C_{n+1})$ with $L(w) =m \ge 2$. 
Fix  a reduced expression of $w$ as follows: 
$$
w =  u_1 t_{n+1} u_2 t_{n+1} \dots u_m t_{n+1} u_{m+1}  $$
with $u_i$, for $1\le i \le m+1$, a reduced expression of a fully commutative element in 
$W^c(B_{n+1})$.  
Then $u_2, \dots, u_m$ are extremal  elements and there is a reduced expression of $w$ of the  form:  
 \begin{equation}\label{forme1}
w = [  i_1,n ]  t_{n+1}  [  i_2,n ] t_{n+1} \dots  [ i_m,n ]  t_{n+1} v_{m+1}  \\
 \end{equation} 
where $ v_{m+1} \in W^c(B_{n+1})$,  $-n  \le i_m \le \dots i_2 \le  i_1 \le n +1$ and $i_2 \le n$.
 \end{lemma} 
 \begin{proof} Since $t_{n+1}$ commutes with $W( B_{n})$, 
  the fact that the expression is reduced forces  $u_i$ to be extremal for $2\le i \le m$. 
We use form~(\ref{Stembridge}) for $u_1$  and  write it as   $u_1=[  l_1,n ] x_1$, a reduced expression with $x_1$ in $W^c( B_{n})$ 
and $-n \le l_1 \le n+1$.  Here $x_1$ commutes with $t_{n+1}$ hence, setting $i_1=l_1$,  we get a reduced expression 
$$
w =[ i_1,n ] t_{n+1} x_1 u_2 t_{n+1} \dots u_m t_{n+1} u_{m+1}  .$$ 
Again $x_1u_2 \in W^c(B_{n+1})$  has a reduced expression $ [  i_2,n ] x_2$ with $-n \le i_2 \le n $ (since 
$x_1u_2$    is extremal)
and $x_2$ in $W( B_{n})$, and this $x_2$ commutes with $t_{n+1}$  and can be pushed to the right, leading to 
$$
w =[ i_1,n ] t_{n+1}[i_2,n ] t_{n+1} x_2u_3 t_{n+1} \dots u_m t_{n+1} u_{m+1}  .$$
 Proceeding from left to right we  obtain formally form (\ref{forme1}). 

Assume $i_{j+1} \ge i_j$ for some $j$, 
$1\le j < m$. If $0 <| i_{j+1}|<n$, the term $\sigma_{| i_{j+1}|}$ on the right of the $j$-th $t_{n+1}$ (starting from the left) can be pushed to the left until we reach the braid  $\sigma_{| i_{j+1}|}\sigma_{| i_{j+1}|+1}\sigma_{| i_{j+1}|}$, a contradiction to the full commutativity. If $ i_{j+1}  =-n$,   we actually have $  i_{j+1}=i_j=-n$. 
 If $  i_{j+1}  =n$,  then $i_j \le n$ and our expression contains the braid $ \sigma_n t_{n+1}\sigma_n t_{n+1}$, again a contradiction. Finally if $  i_{j+1}  =0$, we must have $  i_{j}  =0$ as well because a    negative $i_j$ would  produce, after pushing $\sigma_{i_{j+1}}=\sigma_0$ to the left, the braid $\sigma_1 \sigma_0 \sigma_1 \sigma_0$,  a contradiction. We thus get the inequalities announced.  
  \end{proof}

\begin{lemma}\label{fullandsigma} 
Let $w$ be  a fully commutative element in $W(\tilde C_{n+1})$ with $L(w) =m \ge 2$. Write $w$ in  form  (\ref{forme1})   from  Lemma \ref{lemmafull}. We have:

\begin{enumerate}
\item If $i_s =   -n $ for some $s$ with  $ 1 \le s \le m $, then $i_j =-n $ for    $ 2 \le j \le m $. 
\item If $i_s = 0 $ for some $s$ with  $ 2 \le s \le m $, then $i_j =   0 $ for    $ s \le j \le m $. 
\item If   $ -1 \ge i_s   \ge -n+1 $  for some $s$ with  $  2\le s  \le m $, then  $s=m$ and $v_{m+1} = 1$. 
\item If $i_s >0$  for some $s$ with  $ 1 \le s \le m-1 $, then either $i_s > |i_{s+1}|$, or 
$s=1$ and $ i_{2} =-n$. 
\end{enumerate}
\end{lemma} 

\begin{proof} In case (1) the inequalities in Lemma~\ref{lemmafull} give the result for $j \ge s$. If  
  $s \ge 3$   the reduced expression contains $t_{n+1}[ i_{s-1},n ] t_{n+1}\sigma_n$. If $ i_{s-1}$ was not equal to $-n$, we could push to the right the leftmost term  $t_{n+1}$, which 
 commutes with   $\sigma_j$ for $j<n$, getting a  reduced expression that contains  
$t_{n+1}\sigma_n t_{n+1}\sigma_n$, a contradiction to the full commutativity.  

In case (2), if $s<m$,   we know from  Lemma~\ref{lemmafull} that $ i_{s+1}\le 0$. Our reduced expression contains $t_{n+1}[ 0,n ] t_{n+1}\sigma_{| i_{s+1}|}$. We argue as in the  proof of  Lemma \ref{lemmafull}: a negative $ i_{s+1}$ would produce either the braid $ t_{n+1}\sigma_n t_{n+1}\sigma_n $ or the braid 
  $\sigma_{| i_{s+1}|}\sigma_{|i_{s+1}|+1}\sigma_{| i_{s+1}|}$, a contradiction. 

For  case (3) we observe similarly that   the  expression  $t_{n+1}[ i_{s},n ] t_{n+1}\sigma_j$ would produce the braid 
$\sigma_{j+1} \sigma_{j } \sigma_{j+1} $ 
if $0<j<n$, the braid  $ t_{n+1}\sigma_n t_{n+1}\sigma_n $ if $j=n$, and the braid 
$\sigma_1 \sigma_0 \sigma_1 \sigma_0$ if $j=0$. This  leaves no possibility other than the one announced.  

For case (4), we observe again that $i_s \le |i_{s+1}|< n$  would produce the braid 
$\sigma_{| i_{s+1}|}\sigma_{| i_{s+1}|+1}\sigma_{| i_{s+1}|}$, and $i_{s+1}= n$ 
would produce the braid $\sigma_n t_{n+1}\sigma_n t_{n+1}$. We are left with checking the case $ i_{s+1}= - n$. 
If $s>1$ this  produces the braid  $ t_{n+1}\sigma_n t_{n+1}\sigma_n $ because the $(s-1)$-th $ t_{n+1} $ from the left can be pushed to the right until it reaches $\sigma_n$, whence the result. 
\end{proof}

With these lemmas in hand we are ready to present the classification of fully commutative elements in 
$W (\tilde C_{n+1})$.

\begin{theorem}\label{FC} 
Let $w \in W^c(\tilde C_{n+1})$ with $L(w)   \ge 2$.  
Then $w$ can be written in a unique way as a reduced word of   one and only one of the following two forms, for non negative integers $p$ and $k$:  
\begin{description}
\item[First type]
\begin{equation}\label{formefinalefirsttype}
  w =   [  i,n ]   t_{n+1}( [  -n,n ] \;  t_{n+1} )^k  ([ f,n ] )^{-1} \end{equation}
 with $ -n \le i \le n+1 $ and $-n \le f \le n+1 $.\\ 

\item[Second type]
\begin{equation}\label{formefinalesecondtype}
\begin{aligned}
w &=   [  i_1,n ]  \; t_{n+1}  [  i_2,n ]  \; t_{n+1} \dots  [  i_p,n ]  \;  t_{n+1}    ( [  0,n ]  \; t_{n+1} )^k  
\;    w_r   \quad  \text{if } p>0 ,  \\  
w &=       ( [  0,n ]  \; t_{n+1} )^k  
\;    w_r  \quad   \text{if } p=0, 
 \end{aligned}
 \end{equation}
  with $w_r \in W^c(B_{n+1})$   and  \\

\begin{itemize}
\item if $k > 0$:     $w_r=[ 0,r_1 ][ 0,r_2 ] \dots [ 0,r_u ]$ with $-1 \le r_u < \dots < r_1  \le n$ ; 
\item   if $p > 0$: 
 $ \  
n+1 \ge i_1>... > i_{p-1} > |i_p| >0    \  $; 
\item   if $p > 0$ and   $  i_p  < 0$:   $k=0$,  $w_r=1$ and $i_p \ne -n$;
\item if $k =0$ and $i_p > 0$:  $w_r$ is of form (\ref{Stembridge})   such that $ |l_1 |< i_p$. 
\\

\end{itemize}
\end{description}

The affine length of $w$ of the first (resp. second) type is  $k+1$ (resp. $p + k$) and we have  $ 0 \le p \le n+1  $. \\ 

Now suppose that $L(w) = 1$, then it has a reduced expression of the form: 
 \begin{equation}\label{formefinalelongueur1}
  [  i,n ]  \  t_{n+1}   \    v  
  \end{equation}
 where

\begin{itemize}
\item  if $0 < i \le n+1$ then $v$ is of the form (\ref{Stembridge}) such that for $1\leq j \leq r$ either $l_j= n-j+1$ or $l_j<i$;
\item   if $i <0 $ then $v = ([ h,n ] )^{-1}$ with $-n \le h \le n+1$;
\item  if $  i=0 $ then either $v$ is equal to $ ([ h,n ] )^{-1}$ for $-n \le h \le n+1$, or to 

$([ z,n ])^{-1}[ 0,r_1 ][ 0,r_2 ] \dots [ 0,r_m ]$ for $-1 \le r_m < \dots r_2 < r_1 < z \le n+1$.\\
\end{itemize}

Conversely,   every $w$ of the above form is in $W^c(\tilde C_{n+1})$.

\end{theorem} 
\begin{proof}
We  start   with an  element 
$w \in W^c(\tilde C_{n+1})$ with $L(w) = m \ge 2 $, written  as in (\ref{forme1}), and we discuss according to the value of    $i_2$. 
\begin{enumerate}
\item  If $i_2=-n$, we get from Lemma \ref{fullandsigma} an element of the first type: the same arguments show that any reduced expression of  the rightmost term $v_{m+1}$ must start with  $\sigma_n$  on the left, which forces the shape of $v_{m+1}$.   
\item If   $-(n-1)\le i_2 < 0$ Lemma \ref{fullandsigma} gives directly a second type element with $p=2$, $k=0$ and $w_r=1$. 
\item  If  $i_2=0$  we get  $i_s=0$ for $2 \le s \le m$  from Lemma \ref{fullandsigma}   and  $i_1 \ge 0$ from Lemma 
\ref{lemmafull}.  We thus have an element of the second type, with $p=0$  if $i_1 = 0$, or $p=1$ if $i_1 > 0$. 
Any reduced expression of the rightmost term  $v_{m+1}$ must start with  $\sigma_0$  on the left, which forces the shape of $v_{m+1}=w_r$.   
\item If   $1\le i_2\le n$, then we must have $|i_3| < i_2$ by Lemma \ref{fullandsigma}. We iterate this process 
until we find $j$ with $i_2 > \dots > i_j \ge 1$ and either $j=m$, hence we have an element of the second type with 
$p=m$ and $k=0$, or $i_{j+1} \le 0$, which gives an element of the second type, with $p=j$ and $k>0$ if $i_{j+1} = 0$, or  
$p=j+1$, $k=0$ and $w_r=1$ if $i_{j+1} < 0$ (Lemma \ref{fullandsigma}). When $k$ is positive, $w_r$ is as in case 
(3). When $k=0$ the condition on $w_r$ follows as in the proof of Lemma \ref{fullandsigma} (4). 
\end{enumerate}

For an element $w \in W^c(\tilde C_{n+1})$ of affine length   $L(w) = 1 $, 
written     $w=[  i,n ]  \  t_{n+1}   \    v  $ with $v \in W^c(B_{n+1})$,  the arguments are similar, using   form (\ref{Stembridge}) for $v$. If $i<0$ (resp. $i=0$) 
     any reduced expression of    $v $ must start with  $\sigma_n$ (resp. $\sigma_n$ or $\sigma_0$) on the left. If 
$i=n+1$, there is no further condition on $v$, while if $0<i\le n$  any reduced expression of    $v $ must start with  $\sigma_n$ or $\sigma_t$ with $t <i$.  

The fact that any element of one of these forms is fully commutative is proven by an easy induction. 
\end{proof}

We remark that elements of the first type and elements of affine length $1$ of the form   $[  i,n ]  \  t_{n+1}   ([ h,n ] )^{-1}$ 
  have a unique reduced expression.  Moreover,   an element of affine length at least $2$  has a unique reduced expression if and only if it is of the first type.  
Inserting the elements of affine length   $1$ in the first type and second type sets  would not have given us a partition of the set of those  elements as we will see in the next example. This is the reason  why we handle them separately.

  \begin{example} {\rm We   list the elements in $W^{c}(\tilde{C}_{2} )$ of positive affine length.}  
  \begin{itemize}
\item   First  type elements:  
$$   c \  (\sigma_1 t \sigma_1 t_2)^h \  d \qquad \text{with } \     \left\{\begin{matrix}   h \ge 1,  \cr 
    c  \in  \left\{  1, t_2 ,  \sigma_1 t_2, t \sigma_1 t_2\right\}  ,  \cr
   d  \in  \left\{  1,  \sigma_1, \sigma_1 t,  \sigma_1 t \sigma_1\right\}  , \cr 
\text{if } c=1 , \text{then }h\ge 2.  \end{matrix}\right.
 $$
 
\item   Second type elements:   
$$   a \  ( t \sigma_1t_2)^k \  b    \   \qquad \text{with } \     \left\{\begin{matrix} k \ge 1, \cr 
   a \in \left\{ 1,  t_2,   \sigma_1t_2\right\}, \cr  
    b \in  \left\{  1,  \sigma_1, t,  \sigma_1t, t \sigma_1 \right\}, \cr 
\text{if } a=1 , \text{then }k\ge 2.  \end{matrix}\right.
 $$

\item  Elements of affine length $1$:   
$$  \qquad   e t_2 f   \quad \text{with } \     \left\{\begin{matrix} 
\text{either }  e =  \sigma_1 t  \sigma_1 \text{ and }   f \in  \left\{ 1, \sigma_1, t\sigma_1 ,  \sigma_1 t  \sigma_1\right\}, \qquad \qquad 
 \cr 
\text{ or } e \in \left\{ 1, \sigma_1, t \sigma_1 \right\} \text{ and }  f \in  \left\{   1, t, t \sigma_1, t\sigma_1t,\sigma_1,  \sigma_1 t,  \sigma_1 t  \sigma_1 \right\}. 
\end{matrix}\right.$$
  \end{itemize}

{\rm Notice that if $h$ and $k$ were allowed to be null,  then  $\sigma_1 t_2 \sigma_1$ could be obtained in two different ways:  $ a= \sigma_1t_2, b= \sigma_1$, $k=0$, and 
$ c = \sigma_1t_2 , d= \sigma_1 $, $h=0$.}
    \end{example}

    \begin{remark}\label{CC} {\rm 
In \cite{CFC} the authors define and study {\it cyclically fully commutative}    elements:   elements for which a cyclic permutation of the terms of any reduced expression  transforms it into  a reduced expression for a fully commutative element. 
The normal form given in Theorem \ref{FC} may be used for such a study. 
Indeed let   $w$ be a first type element given in its normal form. Since it has a unique reduced expression, it is easy to see that $w$ is cyclically fully commutative   if and only if either $0 \leq i \leq n+1 $ and $ f= -(i-1)$,  or  $ -n \leq i < 0 $ and $ f = - (i+1)$. In this case,  after an $n-(i+1)$ (first case) or  $n- (i-1)$  (second case) cyclic shift,  $w$ is transformed  into  $( [  -n,n ] \;  t_{n+1} )^{k+1} $.  Let now  $w$ be an element of the second type given in its normal form with $k>0$. Suppose that $w$ is cyclically fully commutative, then $m=p$, moreover:
$$
  w= [  i_1,n ]  \; t_{n+1}  [  i_2,n ]  \;  t_{n+1} \dots  [  i_p,n ]   \; t_{n+1}   ( [  0,n ]  \; t_{n+1} )^k [ 0,i_1 -1 ][ 0,i_2 -1 ] \dots [ 0,i_p-1 ].
$$
   In this case $w$ is transformed into  $( [  0,n ]  \; t_{n+1} )^{k+p}$ by a suitable cyclic shift. }
\end{remark}

\section{The tower of $\tilde C$-type  fully commutative elements} 

The Coxeter group $W(B_n)$ with  Coxeter generators $t,\sigma_1, \cdots, \sigma_{n-1}$,  is  a parabolic subgroup of $W(B_{n+1})$. This is no longer the case for $W(\tilde C_{n })$ and $W(\tilde C_{n+1})$ -- proper parabolic subgroups of $W(\tilde C_{n+1})$ are finite. 
This is an important difficulty when dealing with the affine case. 
As for $W(\tilde C_{n })$,  the injection $P_n: W(\tilde C_{n }) \longrightarrow W(\tilde C_{n+1 })$ of Corollary~\ref{Coxeterinj} 
  is a group monomorphism that preserves the full commutativity of  first type elements and   elements of affine length 1    in $W^c(\tilde C_{n }) $, but  does not preserve it for $t_n[0,n-1]t_n$, for example, in the set of second type fully commutative elements. We will take advantage of the normal form for fully commutative elements established in Theorem \ref{FC}  to produce embeddings from  $W^c(\tilde C_{n}) $ into $W^c(\tilde C_{n+1}) $. \\

 For $ n>0 $, we denote by   $W^c_1(\tilde C_{n}) $   the set of first type fully commutative elements in addition to fully commutative elements  of affine length 1, and  by  $W^c_2(\tilde C_{n}) $   the set of second type fully commutative elements. We thus  have   the following partition: 
$$
W^c(\tilde C_{n}) = W^c_1(\tilde C_{n}) \bigsqcup W^c_2(\tilde C_{n})  \bigsqcup W^c(B_{n}).
$$

\begin{definition}\label{defIJ}  
For any $w\in W^c(\tilde C_{n })$ we define elements $I(w)$   and  $J(w)$ of $W (\tilde C_{n+1} ) $ by the following expressions: 
\begin{itemize}
\item  if  $w \in W^c_2(\tilde C_{n } )$, then $I(w)$  (resp.  $J(w)$) is  obtained by substituting 
$\sigma_n t_{n+1}$ (resp. $t_{n+1} \sigma_n $)  to $t_n$ in the normal form
(\ref{formefinalesecondtype}) for $w$;
\item if  $w \in W^c_1(\tilde C_{n } )$, then $I(w)=J(w)$ is obtained by substituting $\sigma_n t_{n+1}\sigma_n$ to $t_n$ in the normal form (\ref{formefinalefirsttype}) or (\ref{formefinalelongueur1}) for $w$;  
\item if  $w \in W^c(B_{n})$,  then $I(w)=J(w)=w$.
\end{itemize}
\end{definition}

\begin{theorem} \label{IJ} 
For any $w\in W^c(\tilde C_{n })$, the 
   expressions     for $I(w)$ and $J(w)$ in Definition  \ref{defIJ} are reduced and they are reduced expressions for
 fully commutative elements  in $W(\tilde C_{n+1})$. The maps thus defined: 

$$I, J : W^c(\tilde C_{n} ) \longrightarrow W^c(\tilde C_{n+1} ) $$

\medskip\noindent
are injective, preserve the affine length  and satisfy 
  $$ \begin{aligned}
l(I(w)) &=  l(J(w)) = l(w) + L(w) \quad  \text{ for } w \in W^c_2(\tilde C_{n } ),  \\ 
   l(I(w)) &=  l(J(w)) = l(w) +2 L(w)  \quad  \text{ for }  w \in W^c_1(\tilde C_{n } ).
\end{aligned}
$$ 
The injections $I$ and $J$ map first type (resp. second type) elements to  first type (resp. second type) elements 
and their images   intersect exactly on $I(W^c_1(\tilde C_{n})  \bigsqcup W^c(B_{n}))$.  \\ 
\end{theorem}

\begin{proof}  We have for $-(n-1)\le i, f \le n$:   
$$I  ([  i,n-1 ]   t_{n}( [  -(n-1),n-1 ] \;  t_{n} )^k  ([ f,n-1 ] )^{-1})= [  i,n ]   t_{n+1}( [  -n,n ] \;  t_{n+1} )^k  ([ f,n ] )^{-1}$$
hence if 
  $w$ in $W^c (\tilde C_{n  } ) $ is a first type element written in form  (\ref{formefinalefirsttype}), then the expression 
$I(w)=J(w)$ is the normal form  (\ref{formefinalefirsttype}) of 
a first type element in $W^c (\tilde C_{n+1  } ) $. 

Similarly if $w= [  i,n-1 ]   t_{n}v$ is fully commutative of  affine length 1 written in form (\ref{formefinalelongueur1}), we have: 
$$ I(w)= [  i,n-1 ]  \sigma_n  t_{n+1} \sigma_n v =  [  i,n  ]     t_{n+1} (\sigma_n v)$$ 
which is the normal form (\ref{formefinalelongueur1}) of 
an element of affine length $1$  in $W^c (\tilde C_{n+1  } ) $. \\

Now let $w$ be a second  type element written in form  (\ref{formefinalesecondtype}), we see directly that 
$$
\begin{aligned} 
I(w)&=I( [  i_1,n-1 ]  \; t_{n}  [  i_2,n-1 ]  \; t_{n} \dots  [  i_p,n-1 ]  \;  t_{n}    ( [  0,n-1 ]  \; t_{n} )^k  w_r) \\
&=  [  i_1,n ]  \; t_{n+1}  [  i_2,n ]  \; t_{n+1} \dots  [  i_p,n ]  \;  t_{n+1}    ( [  0,n ]  \; t_{n+1} )^k  w_r , 
\end{aligned} 
$$
 the  normal form  (\ref{formefinalesecondtype}) of 
a second type element in $W^c (\tilde C_{n+1  } ) $. We now compute $J(w)$, recalling that 
$t_{n+1} $ commutes with $\sigma_i$ for $0 \le i < n$: 
$$
\begin{aligned} J(w)&=J( [  i_1,n-1 ]  \; t_{n}  [  i_2,n-1 ]  \; t_{n} \dots  [  i_p,n-1 ]  \;  t_{n}    ( [  0,n-1 ]  \; t_{n} )^k  w_r) \\
&=[  i_1,n-1 ]  \; t_{n+1} \sigma_n [  i_2,n-1 ]  \; t_{n+1} \sigma_n \dots  [  i_p,n-1 ]  \;  t_{n+1}  \sigma_n  ( [  0,n-1 ]  \; t_{n+1} \sigma_n)^k  w_r \\  
&=
  \; t_{n+1} [  i_1,n-1 ] \sigma_n \; t_{n+1} [  i_2,n-1 ]  \sigma_n \dots  \;  t_{n+1}  [  i_p,n-1 ]  \sigma_n  (  \; t_{n+1} [  0,n-1 ] \sigma_n)^k  w_r \\
&= \; t_{n+1} [  i_1,n ]  \; t_{n+1} [  i_2,n ]   \dots  \;  t_{n+1}  [  i_p,n ]   (  \; t_{n+1} [  0,n ] )^k  w_r 
\end{aligned} 
$$
If $k=0$, we check that $w'_r=  [  i_p,n ] w_r $ satisfies the condition on the rightmost term in form  (\ref{formefinalesecondtype}). If $k>0$: 
$$
\begin{aligned} J(w)&=
  \; t_{n+1} [  i_1,n ]  \; t_{n+1} [  i_2,n ]   \dots  \;  t_{n+1}  [  i_p,n ]   t_{n+1} (  \; [  0,n ]  t_{n+1})^{k-1} [  0,n ] w_r  , 
\end{aligned} 
$$
where again  $w'_r=[  0,n ] w_r$ satisfies the condition Theorem \ref{FC}. Hence  $J(w)$ is the normal form 
of a second type element in $W^c (\tilde C_{n+1  } ) $. 
 We notice that the leftmost term in the expression of $J(w)$ is  $t_{n+1}$ whereas no reduced expression of $I(w)$ can have $t_{n+1}$ as its leftmost term (since the expression $ t_{n+1}I(w)$ is also a normal form of the second type), therefore the images  $I(W^c_2(\tilde C_{n } ))$ and $J(W^c_2(\tilde C_{n } ))$ are disjoint.  \\

$I$ and $J$ clearly preserve the affine length. 
The fact that   the substitution process on $W^c_2(\tilde C_{n}) $ (resp. on $W^c_1(\tilde C_{n}) $) adds to the original length the number of occurrences (resp. the double of the number of occurrences) of $t_n$, i.e. the affine length or its double, is clear.  The injectivity of both maps results from the uniqueness of the normal form. 
\end{proof} 

\medskip
We remark that the injections $I$ and $J$ on $W^c_1(\tilde C_{n})  \bigsqcup W^c(B_{n})$ are   but the restriction   of $R_n$. Actually $I$ and $J$ may be defined on all   $ W(\tilde C_{n }) $, but as we don't need this, we won't examine it further.

\section{The tower of  $\tilde C$-type Temperley-Lieb algebras}\label{TL}  

Let $K$ be an integral domain  of characteristic $0$ and let $q$ be an invertible element in $K$. Let  $\Gamma$ be a Coxeter graph with associated Coxeter system $(W,S)$ and    Hecke algebra  $H\Gamma(q)$.  Following Graham \cite[Definition 6.1]{Graham}, 
we define the $\Gamma$-type Temperley-Lieb algebra $TL\Gamma(q)$ to be the quotient of the Hecke algebra 
 $H\Gamma(q)$ by the two-sided ideal generated by the elements $ L_{s,t} = \sum_{w\in <s,t>} g_w $, 
where $s$ and $ t  $ are non commuting elements in $S $ such that $st$ has finite order.  For $w$ in $W$ we denote by $T_w$ the image of $g_w \in H\Gamma(q)$ under the canonical surjection from $H\Gamma(q)$ onto $TL\Gamma(q) $.  The set $\left\{ T_w \ | \ w \in W^c(\Gamma) \right\}$ forms a  $K$-basis for $TL\Gamma(q)$ \cite[Theorem 6.2]{Graham}. \\

	 For $x,y$ in a given ring with  identity,  we define: 
$$\begin{aligned}
V(x,y) &= xyx+xy+yx+x+y+1, \\ 
  Z(x,y)   &=  xyxy+xyx+yxy+xy+yx+x+y+1
.
\end{aligned}
$$
		
			For $ n\geq 1$, the $\tilde{C}$-type Temperley-Lieb algebra with $n+2$ generators $TL\tilde{C}_{n+1}(q)$ is given by the   set of generators $\left\{ T_t, T_{\sigma_{1}}, \dots T_{\sigma_{n}}, T_{t_{n+1}}\right\}$, with the defining relations (with our convention $t= \sigma_0$):\\
	 	\begin{equation}\label{definingrelations} 	
	   \quad \left\{ \quad  \begin{aligned}
		 &T_{\sigma_{i}} T_{\sigma_{j}} =T_{\sigma_{j}} T_{\sigma_{i}}  \text{ for } 0 \leq i,j\leq n \text{ and } \left| i-j\right| \geq 2, \\
		  &T_{\sigma_{i}} T_{t_{n+1}} =T_{t_{n+1}} T_{\sigma_{i}}   \text{ for  }  0\leq i \leq n-1, \\
	          &T_{\sigma_{i}}T_{\sigma_{i+1}}T_{\sigma_{i}} = T_{\sigma_{i+1}}T_{\sigma_{i}}T_{\sigma_{i+1}}  \text{ for }  1\leq i\leq n-1,\\
         	    &T_{t}T_{\sigma_{1}}T_{t}T_{\sigma_{1}} =T_{\sigma_{1}}  T_{t} T_{\sigma_{1}} T_{t}   ,\\
			    &T_{t_{n+1}}T_{\sigma_{n}}T_{t_{n+1}}T_{\sigma_{n}} =T_{\sigma_{n}}  T_{t_{n+1}} T_{\sigma_{n}}  T_{t_{n+1}}  ,\\
			    &T^{2}_{t_{n+1}} = (q-1)T_{t_{n+1}} +q \text{ and } T^{2}_{\sigma_{i}} = (q-1)T_{\sigma_{i}} +q  \text{ for }  0\leq i\leq n,\\
						  &V(T_{\sigma_{i}},T_{\sigma_{i+1}}) = 0  \text{  for }1\leq i\leq n-1, \\
						  & Z(T_{\sigma_{1}},T_{t}) = Z(T_{\sigma_{n}},T_{t_{n+1}})= 0  .
			  \end{aligned} \right.     \qquad 
		\end{equation}
 
 We set $TL\tilde{C}_{1}(q)= K$. In the following we denote by 
$h_w$, $w \in W^c(\tilde{C}_n)$, the basis elements of  $TL\tilde{C}_{n}(q)$ to distinguish them from those of $TL\tilde{C}_{n+1}(q)$.

		                     
                    \begin{lemma}\label{morphismFn} 
The morphism of algebras $R_n: H\tilde{C}_n (q)   \longrightarrow   H\tilde{C}_{n+1} (q) $ defined 
in (\ref{defFn})   induces the following morphism of algebras, which we also denote by  $R_n$:	
			\begin{eqnarray}
				R_{n}: TL\tilde{C}_{n}(q) &\longrightarrow& TL\tilde{C}_{n+1}(q) \nonumber\\
				h_{\sigma_{i}} &\longmapsto & T_{\sigma_{i}} \  \   \text{ for } 0 \leq i \leq n-1 \nonumber\\
				h_{t_{n}} &\longmapsto & T_{\sigma_{n}} T_{t_{n+1}} T^{-1}_{\sigma_{n}}. \nonumber
			\end{eqnarray}
The restriction of $R_n$ to $  TLB_{n }(q)$ is an injective morphism into $TLB_{n+1}(q)$  and satisfies $R_n(h_w)= g_{I(w)}= g_{J(w)}$  for  $w \in W^c( B_{n})$. 
                    \end{lemma}
                    
                    \begin{proof}
                    
                    The lemma follows after noticing that
                    \begin{eqnarray}
                     Z(R_n(h_{\sigma_{n-1}}),R_n(h_{t_{n}})) = (T_{\sigma_{n}} T_{\sigma_{n-1}})  Z (T_{\sigma_{n}},T_{t_{n+1}})(T_{\sigma_{n}} T_{\sigma_{n-1}})^{-1}.\nonumber
                    \end{eqnarray}
                    \end{proof}
		 
		 The aim of this section is to show, using  the normal form  of Theorem \ref{FC},   that the morphism $R_n$ is an  injection.  
We  set $p=1/q$. We will use repeatedly 
\begin{equation}\label{basic} 
 R_n(h_{t_{n}}) =   p T_{\sigma_{n}t_{n+1}\sigma_{n}}+ 
  (p-1) T_{\sigma_{n}t_{n+1}},  
\end{equation}
as well as the following consequences of the defining relations (\ref{definingrelations}): \\ 

\noindent 
(i) 
In $TL\tilde{C}_{n+1}(q)$, a product $T_w T_y$, $w, y \in W^c(\tilde{C}_{n+1})$,   is a linear combination of terms $T_z$, $z\in W^c(\tilde{C}_{n+1})$,  with 
$L(z) \le L(w)+L(y)$  and $l(z) \le l(w)+l(y)$. \\ 

\noindent 
(ii) 
 When   a braid $T_{\sigma_{i}}T_{\sigma_{i+1}}T_{\sigma_{i}}$ appears in a computation, the use of 
$V(T_{\sigma_{i}},T_{\sigma_{i+1}})=0$ replaces it by a sum of terms $T_z$  with  $l(z) =2$, $1$ or $0$, hence the length decreases. \\ 

\noindent 
(iii)  When a braid $T_{t_{n+1}}T_{\sigma_{n}}T_{t_{n+1}}T_{\sigma_{n}}$  occurs, the use of 
$Z(T_{\sigma_{n}},T_{t_{n+1}})= 0 $ replaces it by a sum of terms $T_z$  with  $l(z) \le 3$  in which only one has affine length $2$, namely  $T_{t_{n+1}}T_{\sigma_{n}}T_{t_{n+1}} $. The other terms have affine length 
$1$ or $0$ and will be ignored since we will be interested in those terms with maximal affine length. \\

\begin{lemma}\label{formula} 
 
If  $ w \in W^c_1(\tilde C_{n}) \bigsqcup W^c(B_{n}) $ we have: 
$$
R_n(h_w) =    p^{L(w)}   T_{I(w)} 
+   \sum_{\begin{smallmatrix}L(x)\le L(w)\\ l(x)<l(I(w))\end{smallmatrix}}  \alpha_x T_x  \qquad (\alpha_x \in K). 
$$

\end{lemma}

\begin{proof} We prove the statement by induction on the affine length 
$L(w)$ of $w$, recalling that for $w \in W^c(B_n)$ we have $R_n(h_w) = T_{I(w)}$, which implies the assertion for affine length $0$. 
 We now assume that the property holds for any $u$ of   affine length at most~$k$. \\ 

 Let $w $ in $ W^c _1(\tilde C_{n})$  have affine length  $L(w)= k+1$ and write $w$ in its  normal form given by 
 Theorem \ref{FC} so that  
$ w= u t_n v$ where $u $ is an element in $\in   W^c_1(\tilde C_{n}) \bigsqcup W^c(B_{n})  $ written in    normal form, $L(u)= k $,  
$v \in W^c( B_{n })$ and $l(w) = l(u)+ l( v) +1$. We have: 
 $$I(w)=J(w)= I(u) \sigma_{n}t_{n+1}\sigma_{n} I(v).
$$ 
Since $R_n$ is an homomorphism of algebras we have, using 
(\ref{basic}): 
\begin{eqnarray} 
  R_n (h_w)   =   R_n(h_u)  R_n(h_{t_n})  R_n(h_v)  	 
= R_n(h_u) \   
[ \;  p T_{\sigma_{n}t_{n+1}\sigma_{n}}+ 
  (p-1) T_{\sigma_{n}t_{n+1}}] 
\    T_{I(v)}. 
  \nonumber  
   \end{eqnarray}

 By the induction hypothesis,  
$R_n(h_u)$  is a linear combination of terms $T_z$ with $L(z) \le k$, and    has a unique term of maximal length 
(that is, a term $T_z$ where $z$ has maximal length), which is  $p^{L(u)}T_{I(u)}$.  
Recalling (i) above, we deduce that  $ R_n (h_w)$ is a  linear combination of terms $T_y$ with $L(y) \le k+1$, and 
has a unique term of maximal length which is  $  p^{L(u)} T_{I(u)} p T_{\sigma_{n}t_{n+1}\sigma_{n}} T_{I(v)}= p^{L(w)} T_{I(w)}$. 
   \end{proof} 

 \begin{proposition}\label{formula2}
 
  Let $w$ be in $W^c_2(\tilde C_{n}) $, then for some $\alpha_x, \beta_y  \in K$ we have
 
 $$
R_n(h_w) =   (-1)^{L(w)}   T_{I(w)} +  (-p)^{L(w)}   T_{J(w)} +  \sum_{\begin{smallmatrix}L(y)=L(w)\\ l(y)<l(I(w))\end{smallmatrix}}  \beta_y T_y 
+  \sum_{L(x)< L(I(w))}  \alpha_x T_x.
$$

 \end{proposition}

\begin{proof} We use the normal form (\ref{formefinalesecondtype}) of  Theorem \ref{FC} for $w \in W^c_2(\tilde C_{n}) $ and we remark first  that it is enough to prove our assertion for $w_r=1$ and $i_1=n$. Indeed, the normal form for $w$ is  $w = [  i_1,n-1 ] u w_r$ where $u$ is an element of $W^c_2(\tilde C_{n}) $  having same affine length as $w$, whose normal form begins and ends with $t_n$. Then, assuming our result holds for $u$, we have:  
$$
\begin{aligned}
R_n(h_w)   = &R_n(h_{ [  i_1,n-1 ]})  R_n(h_{u}) R_n(h_{w_r}) \\  = 
T_{ [  i_1,n-1 ] } &((-1)^{L(u)}   T_{I(u)} +  (-p)^{L(u)}   T_{J(u)} +  \sum_{\begin{smallmatrix}L(y)=L(u)\\ l(y)<l(I(u))\end{smallmatrix}}  \beta_y T_y 
+  \sum_{L(x)< L(I(u))}  \alpha_x T_x )  T_{I(w_r)},  
\end{aligned} 
$$ 
hence our result since $I(w)=  [  i_1,n-1 ]I( u) I( w_r)$ and $J(w)=  [  i_1,n-1 ]J( u) J( w_r)$, both expressions reduced and fully commutative, and the remaining terms will have either the same affine length as $I(w)$ but a smaller Coxeter length, or a smaller affine length. \\

We work by induction on the affine length and remark once and for all that the development of $  R_n (h_w) $ 
will contain only terms $T_z$ with $L(z) \le L(w)$ (see (i) above). To prove our claim, we will then focus on terms of affine length $L(w)$. We will actually prove by induction the following more precise statement: \\

{\it   Let $w$ be in $W^c_2(\tilde C_{n}) $, whose normal form  (\ref{formefinalesecondtype})  begins and ends with $t_n$. Then for some $\alpha_x, \beta_y  \in K$ we have: }  
\begin{equation}\label{precise} 
\begin{aligned} 
R_n(h_w) =   (-1)^{L(w)}   &T_{I(w)} +  (-p)^{L(w)}   T_{J(w)} 
 \\ &+  T_{t_{n+1}}T_{\sigma_{n}}T_{t_{n+1}} \   (\sum_{\begin{smallmatrix}L(y)=L(w)-2\\ l(y)<l(I(w))-3 \end{smallmatrix}}  \beta_y T_y  )
+  \sum_{L(x)< L(I(w))}  \alpha_x T_x.
\end{aligned} 
\end{equation}

We start with an element 
  $w=t_n  [  i,n-1 ] t_n$ for $-(n-1) < i \leq n-1$.  Using 
(\ref{basic}) we see that  $R_n(h_w)= R_n(h_{t_{n}})R_n(h_{ [  i,n-1 ] })  R_n(h_{t_{n}})  $ is a linear combination of elements of the basis of  $TL\tilde{C}_{n+1}(q)$ with affine length at most $2$, appearing in  the following products:

 \begin{figure}[ht]
				\centering
				\begin{tikzpicture}
               \begin{scope}[xscale = 1]

  \node at (-4.5,1)  {$[1]~(p-1)   \   T_{\sigma_{n}t_{n+1}}$};
  \node at (-4.9,-1) {$[2]~p \  T_{\sigma_{n}t_{n+1}\sigma_{n}}$};

\draw (-1.5,0)  -- (-2.5,-1);
	\draw (-1.5,0)  -- (-2.5,1);
	
  \node at (0,0)  {$T_{[  i, n-1 ]} $};

	\draw (1.5,0)  -- (2.5,-1);
	\draw (1.5,0)  -- (2.5,1);

  \node at (4.5,1)  {$[1']~(p-1)  \   T_{\sigma_{n}t_{n+1}}$  };
  \node at (4.1,-1) {$[2']~p \ T_{\sigma_{n}t_{n+1}\sigma_{n}}$};
 
\end{scope}
               \end{tikzpicture}
							
			\end{figure}	
The product associated to [1] and [1'] is $(p-1)^2T_{I(w)}$.

 Developping the one associated to  [1] and [2'], since $T_{ t_{n+1}}$ commutes with $T_{[  i, n-1 ]} $, 
 we find the braid 
\begin{equation}\label{braid4}
T_{t_{n+1}}T_{\sigma_{n}}T_{t_{n+1}} T_{\sigma_{n}}= - T_{t_{n+1}}T_{\sigma_{n}}T_{t_{n+1}} 
-T_{\sigma_{n}}T_{t_{n+1}} T_{\sigma_{n}} - T_{t_{n+1}} T_{\sigma_{n}}
-T_{\sigma_{n}}T_{t_{n+1}}  - T_{t_{n+1}} - T_{\sigma_{n}} -1 
\end{equation}
 so, recalling (iii) above, we get the term $-p(p-1) T_{I(w)}$ and terms of affine length at most $1$.

 Developping the product corresponding to  [2] and [1'] we find the braid 
\begin{equation}\label{braid3}
T_{\sigma_{n}}T_{\sigma_{n-1}}T_{\sigma_{n}}
=  -T_{\sigma_{n}}T_{\sigma_{n-1}} - T_{\sigma_{n-1}}T_{\sigma_{n}}-  T_{\sigma_{n}}
-T_{\sigma_{n-1}} -1. 
\end{equation} 	     
The first term here leads to the braid $T_{\sigma_{n}}T_{t_{n+1}}T_{\sigma_{n}}T_{t_{n+1}} $, hence to only one term of affine length $2$, namely 
$p(p-1)T_{t_{n+1}}T_{\sigma_{n}}T_{t_{n+1}} T_{[  i, n-1 ]}$ which has length  $l(I(w))-1$. The second leads to 
$-p(p-1) T_{I(w)}$. The third $-  T_{\sigma_{n}}$  leads to the braid  (\ref{braid4}) again, of which we keep only the first term, 
hence a  
   $p(p-1)T_{t_{n+1}}T_{\sigma_{n}}T_{t_{n+1}} T_{[  i, n-2 ]}$ which has length  $l(I(w))-2$. 
The fourth and the fifth terms lead to  non reduced expressions  containing $T_{t_{n+1}}^2$, hence to terms of   affine length  strictly less than $L(w)$.

 Similarly, developping the product with   
 [2] and [2'], we find the braid $T_{\sigma_{n}}T_{\sigma_{n-1}}T_{\sigma_{n}}$. The first term in (\ref{braid3}) leads to the braid  $T_{\sigma_{n}}T_{t_{n+1}}T_{\sigma_{n}}T_{t_{n+1}} $ of which we keep 
the term $-T_{t_{n+1}}T_{\sigma_{n}}T_{t_{n+1}}$, leading to a final term $ p^2 T_{J(w)}$. 
The second term in (\ref{braid3}) leads to $T_{t_{n+1}}T_{\sigma_{n}}T_{t_{n+1}} T_{\sigma_{n}}$ of which we keep as above  $-T_{t_{n+1}}T_{\sigma_{n}}T_{t_{n+1}}$, leading to $ p^2 T_{I(w)}$.
 The third $-  T_{\sigma_{n}}$  leads to the braid  (\ref{braid4}), of which we keep the first term, 
hence a   $p^2  T_{t_{n+1}}T_{\sigma_{n}}T_{t_{n+1}} T_{[  i, n-2 ]} T_{\sigma_{n}}$ where 
$T_{[  i, n-2 ]} T_{\sigma_{n}}$ has length 
at most  $l(I(w))-4$ (we could  reduce this term again with (\ref{braid4}) but this is the form suited for our purpose). With the fourth and fifth terms we get as before a $T_{t_{n+1}}^2$, decreasing the affine length.  

    Summing up we get  our 
		  two elements: 
$$((p-1)^2+p^2+2p(1-p))T_{I(w)} = T_{I(w)} \quad \text{ and } \quad p^{2}T_{J(w)},  $$ 
plus a product of $ \  T_{t_{n+1}}T_{\sigma_{n}}T_{t_{n+1}}$ by terms of affine length $0$ and Coxeter length at most $l(I(w))-4$, plus 
   terms of smaller affine length, as announced.  \\

 Now that we have detailed the case of affine length $2 $, we will be quicker with the induction step. Indeed 
take  $w$   in $W^c_2(\tilde C_{n}) $, in its normal form  (\ref{formefinalesecondtype})  beginning and ending with $t_n$, of affine length at least $3$, and write it    $w= t_n[  i ,n-1 ] u$ where $-(n-1) < i \leq n-1$ and 
$u$ satisfies the same conditions except that $L(u) \ge 2$. Assuming that our statement holds for $u$ we write: 

\begin{equation} 
\begin{aligned} 
R_n(h_w) &=  R_n(h_{t_{n}}) R_n( h_{ [  i,n-1 ] })   R_n(h_u)   \\ 
 &= ( p T_{\sigma_{n}t_{n+1}\sigma_{n}} + 
   (p-1) T_{\sigma_{n}t_{n+1}}) \   T_{[  i, n-1 ]}  \    [ 
  (-1)^{L(u)}    T_{I(u)} 
 \\ &   +  (-p)^{L(u)}   T_{J(u)}  +  T_{t_{n+1}}T_{\sigma_{n}}T_{t_{n+1}} \   (\sum_{\begin{smallmatrix}L(y)=L(u)-2\\ l(y)<l(I(u))-3 \end{smallmatrix}}  \beta_y T_y  )
+  \sum_{L(x)< L(I(u))}  \alpha_x T_x  ] 
\end{aligned} 
\end{equation}
and examine the terms of this product, ignoring the sum over $x$ with  $L(x)< L(I(u))$ that will anyway lead 
to terms $T_z$ with $L(z)< L(I(u))$ (rule (i)). We draw as before: 
\begin{figure}[ht]
				\centering
				\begin{tikzpicture} 
               \begin{scope}[xscale = 1]

  \node at (-4.5,1)  {$[1]~(p-1)   \   T_{\sigma_{n}t_{n+1}}$};
  \node at (-4.9,-1) {$[2]~p \  T_{\sigma_{n}t_{n+1}\sigma_{n}}$};

\draw (-1.5,0)  -- (-2.5,-1);
	\draw (-1.5,0)  -- (-2.5,1);
	
  \node at (0,0)  {$T_{[  i, n-1 ]} $};

	\draw (1.5,0)  -- (2.5,-1);
	\draw (1.5,0)  -- (2.5,0);
	\draw (1.5,0)  -- (2.5,1);

  \node at (4.5,1)  {$[1']~  (-1)^{L(u)}    T_{I(u)} $  };

  \node at (4.5,0)  {$[2']~ (-p)^{L(u)}   T_{J(u)}$};

  \node at (4.5,-1) {$[3']~   T_{t_{n+1}}T_{\sigma_{n}}T_{t_{n+1}}$};
 
\end{scope}
               \end{tikzpicture}
							
			\end{figure}	

 [1] and [3']:  we have a $T_{t_{n+1}}^2$ so the affine length drops. 

  [2] and [3']:  we have the braid (\ref{braid4}) which we replace  by $ -  T_{t_{n+1}}T_{\sigma_{n}}T_{t_{n+1}} $, leading to products of $T_{t_{n+1}}T_{\sigma_{n}}T_{t_{n+1}}$ with terms coming from 
$    T_{[  i, n-1 ]}T_{\sigma_{n}}T_{t_{n+1}} T_y $  
hence terms $T_z$ with $L(z) \le L(w)-2$ and $l(z) < l(I(u))-3 + l([  i, n-1 ] {\sigma_{n}} t_{n+1}) = l(I(w)-3$. 

[1] and [2']:   $J(u) $ is fully commutative with a reduced expression starting with  $  {t_{n+1}}$, 
hence  $T_{J(u)} = T_{t_{n+1}} T_v$ and we get a $T_{t_{n+1}}^2$  dropping the affine length.

[2] and [2']: we write again $T_{J(u)} = T_{t_{n+1}} T_v$ and get the braid  (\ref{braid4}) which we replace  by $ -  T_{t_{n+1}}T_{\sigma_{n}}T_{t_{n+1}} $. We thus get: 
$$(-p)(-p)^{L(u)}  T_{t_{n+1}}T_{\sigma_{n}}T_{t_{n+1}} T_{[  i, n-1 ]} T_v = (-p)^{L(w)} T_{J(w)}.$$

[1] and [1']: this gives directly $ (p-1) (-1)^{L(u)}    T_{I(w)} $. 

[2] and [1']: $I(u) $ is fully commutative with a reduced expression starting with  $ \sigma_n {t_{n+1}}$, we write then $T_{I(u)} = T_{\sigma_{n}} T_{t_{n+1}} T_s$ and get the braid (\ref{braid3}). The first term 
 $-T_{\sigma_{n}}T_{\sigma_{n-1}}$ in (\ref{braid3}), leads to (\ref{braid4})  which we replace  by $ -  T_{t_{n+1}}T_{\sigma_{n}}T_{t_{n+1}} $, itself multiplied on the right by 
$T_{[  i, n-1 ]}    T_s$ that develops as a sum of terms $T_z$ of affine length at most $ L(I(u))-1=L(I(w))-2 $ as expected.  The second  term 
 $-T_{\sigma_{n-1}} T_{\sigma_{n}} $ in (\ref{braid3}) provides the term 
$( -p)  (-1)^{L(u)}    T_{I(w)}$ that, added with the one from ([1] and [1']), gives us the 
$ (-1)^{L(w)}    T_{I(w)} $ that we need. 
The third term $-  T_{\sigma_{n}} $ in (\ref{braid3}) again leads to (\ref{braid4}) hence a term starting with $  T_{t_{n+1}}T_{\sigma_{n}}T_{t_{n+1}} $. The fourth and fifth will provide a $T_{t_{n+1}}^2$ dropping the affine length.
\end{proof}

 \medskip

\begin{theorem}\label{R} 
The tower of affine Temperley-Lieb  algebras

\begin{eqnarray}
			  TL\tilde{C}_{1}(q)  \stackrel{R_{1}}{\longrightarrow}  TL\tilde{C}_{2}(q) \stackrel{R_{2}} {\longrightarrow} TL\tilde{C}_{3}(q) \longrightarrow \cdots  TL\tilde{C}_{n}(q)\stackrel{R_{n}} {\longrightarrow}  TL\tilde{C}_{n+1}(q)\longrightarrow  \cdots  \nonumber\\\nonumber
		\end{eqnarray}
	
	is a tower of faithful arrows. 	\\  
\end{theorem}

\begin{proof} We need to show that $R_n$ is an injective homomorphism of algebras. 
A basis for $TL\tilde{C}_{n}(q)$ is given by the elements $h_w$ where $w$ runs over 
$ W^c(\tilde C_{n })$. Assume there are non trivial dependence relations between the images of these basis elements. Pick one such relation, say 
\begin{equation}\label{eqn:dr}
\sum_w \lambda_w R_n(h_w) =0 , 
\end{equation}
 and let 
$m= \max \{ L(w)   \  | \ 
w \in W^c(\tilde C_{n }) \text{ and } \lambda_w \ne 0 \}$. If $m=0$, all elements $w$ such that 
$\lambda_w \ne 0$ belong to $ W^c(B_{n })$,  contradicting the injectivity of 
  the restriction of $R_n$ to $  TLB_{n }(q)$ (Lemma~\ref{morphismFn}). 
So $m$ is positive. 

We know from  Lemma \ref{formula}  and Proposition \ref{formula2} that 
for $w \in  W^c(\tilde C_{n })$, 
$R_n(h_w) $ is a linear combination of terms $T_z$ where $z$ has affine length at most $L(w)$. Therefore the terms  $T_z$ where $z$ has maximal affine length in the development of \eqref{eqn:dr} are exactly the terms of maximal affine length in the development of $ \  \sum_w \lambda_w R_n(h_w) $ where $w$ runs over the set of fully commutative elements of affine length $m$ in $ W (\tilde C_{n })$.   Among them,    Lemma \ref{formula}  and Proposition \ref{formula2} give us 
 the terms of maximal Coxeter length, hence   the terms of maximal Coxeter length among the terms of maximal affine length  $m$  in the development of \eqref{eqn:dr} are  the terms of maximal Coxeter length     in the sum: 
\begin{equation}\label{eqn:fin}
  \sum_{\begin{smallmatrix}   x \in W^c_1(\tilde C_{n}) \\ L(x) = m  \\ \alpha_x \ne 0
\end{smallmatrix} }   \alpha_x   T_{I(x)}   \ 
 +  \sum_{\begin{smallmatrix}   x \in  W^c_2(\tilde C_{n}) \\ L(x)=m \\ \alpha_x \ne 0
\end{smallmatrix} }  \alpha_x ( (-1)^{m}   T_{I(x)} +   (-p)^{m} T_{J(x)}  )    
\end{equation}
The elements $T_d$ for 
 $d  \in  W^c(\tilde C_{n+1})$ form a basis of $\widehat{TL}_{n+1}(q)$. Moreover 
 $I$ and $J$ are injective, map $W^c_i(\tilde C_{n})$ to $W^c_i(\tilde C_{n+1})$ for $i=1, 2$,  and 
$I(W^c_2(\tilde C_{n}))$ and $J(W^c_2(\tilde C_{n}))$ are disjoint  
(Theorem \ref{IJ}). Therefore  
 we see that all the coefficients $\alpha_x$ for $L(x)=m$ and $l(I(x)$ maximal in \eqref{eqn:fin}  must be $0$, a contradiction. 
\end{proof}

		\vspace{1cm}
 
\renewcommand{\refname}{REFERENCES}


\begin{thebibliography}{}
\bibitem{Sadek_Thesis} S. Al Harbat. On the affine braid group, affine Temperley-Lieb algebra and Markov trace. PH.D
Thesis, Universit\'e Paris-Diderot-Paris 7, 2013. \label{Sadek_Thesis}


\bibitem{Sadek_2013_2} S. Al Harbat. Markov trace on a tower of affine Temperley-Lieb algebras of type $\tilde{A }$. 
  J.   Knot Theory Ramifications   24 (2015) 1--28. 
DOI : 10.1142/S0218216515500492



\bibitem{Sadek_2015}  S. Al Harbat.   
    Tower of fully commutative elements of type $\tilde A$ and applications.   arXiv:1507.01521 (2015) 1--23. \label{Sadek_2015}



 

\bibitem{CFC} T. Boothby,  J. Burkert, M. Eichwald, D. C. Ernst, R. M. Green, M. Macauley. On the cyclically fully commutative elements of Coxeter groups. J. Algebraic Combin. 36 (2012), no. 1, 123--148.\label{CFC}

\bibitem{Bourbaki_1981} N. Bourbaki. Groupes et alg\`ebres de Lie, chapitres 4, 5, 6. Masson, 1981. \label{Bourbaki_1981}

\bibitem{Digne_2012} F. Digne. A Garside presentation for Artin-Tits groups of type  $\tilde C_n$. Ann. Inst. Fourier, Grenoble. 62, 2 (2012) 641-666. \label{Digne_2012}

\bibitem{Ernst_2008} D.C. Ernst. A diagrammatic representation of an affine C Temperley-Lieb algebra. Ph.D. Thesis, University of Colorado at Boulder, 2008. \label{Ernst_2008}

\bibitem{Ernst_2012}  D.C. Ernst. Diagram calculus for a type affine C Temperley-Lieb algebra, I. J. Pure Appl. Algebra 216 (2012), no. 11, 2467--2488.

\bibitem{Fan_1995}  C.K. Fan. A Hecke Algebra Quotient and Properties of Commutative Elements of a Weyl group.  Ph.D. Thesis, MIT, 1995.\label{Fan_1995}

 



\bibitem{Graham} J.J. Graham. Modular representations of Hecke algebras and related algebras. Ph.D. Thesis, University of Sydney, 1995.\label{Graham}

 
    

\bibitem{Green2006} R.M. Green. Star reducible Coxeter groups. Glasg. Math. J. 48 (2006), no. 3, 583--609.\label{Green2006}

\bibitem{Green2007}  R.M. Green. Generalized Jones traces and Kazhdan-Lusztig  bases. J. Pure Appl. Algebra 211 (3) (2007) 744--772.\label{Green2007}

\bibitem{HJ} C.R.H. Hanusa, B.C. Jones. The enumeration of fully
commutative affine permutations. European J. Combin. 31 (2010), no. 5,
134--1359.\label{HJ}

\bibitem{Hee} J.-Y. H\'ee.  Syst\`eme de racines sur un anneau commutatif totalement ordonn\'e.   Geom. Dedicata 37 (1991), no. 1, 6--102.

\bibitem{J87}
V. F. R. Jones. 
Hecke algebra representations of braid groups and link polynomials.
Ann. of Math. (2) 126 (1987), no. 2, 335--388. 
 

\bibitem{St96} J. R. Stembridge. On the fully commutative elements of Coxeter groups. J. Algebraic Combin. 5 (4) (1996) 353--385.\label{St96}

\bibitem{St} J. R. Stembridge. Some combinatorial aspects of reduced words in finite Coxeter groups. Transactions A.M.S. 349 (4)  1997, 1285--1332.\label{St}



\end{thebibliography}
\end{document}